\documentclass{article}

\usepackage[english]{babel}

\usepackage[letterpaper,top=2cm,bottom=2cm,left=3cm,right=3cm,marginparwidth=1.75cm]{geometry}

\usepackage{amsmath}
\usepackage{amsthm}
\usepackage{graphicx}
\usepackage[colorlinks=true, allcolors=blue]{hyperref}
\usepackage{amssymb}
\usepackage{tikz}
\usetikzlibrary{positioning}
\usepackage[utf8]{inputenc}
\usepackage[T1]{fontenc}
\usepackage{tcolorbox}
\usepackage{enumitem}
\usepackage{multicol}
\renewcommand{\arraystretch}{1.5} 
\usepackage[table]{xcolor} 
\usepackage{array} 
\usetikzlibrary{arrows}
\usetikzlibrary{arrows.meta}
\usepackage{cmap}
\usepackage{subcaption}

\usepackage{svg}

\newtheorem{theorem}{Theorem}[section]
\newtheorem{corollary}[theorem]{Corollary}
\newtheorem{conjecture}[theorem]{Conjecture}
\newtheorem{lemma}[theorem]{Lemma}
\newtheorem{problem}[theorem]{Problem}
\newtheorem{question}[theorem]{Question}

\newcounter{example}
\renewcommand{\theexample}{\Alph{example}}

\usepackage{xcolor}

\newcommand{\vu}[0]{\mathbf{u}}
\newcommand{\vv}[0]{\mathbf{v}}

\newcommand{\vn}[0]{\mathbf{n}}
\newcommand{\vc}[0]{\mathbf{c}}

\newcommand{\vx}[0]{\mathbf{x}}
\newcommand{\vy}[0]{\mathbf{y}}
\newcommand{\R}[0]{\mathbb{R}}
\newcommand{\Pb}[0]{\mathbb{P}}

\newcommand{\Kc}[0]{\mathcal{K}}

\newcommand{\Bc}[0]{\mathcal{B}}

\newcommand{\skel}[0]{\mathtt{skel}}
\newcommand{\bd}[0]{\mathtt{bd}}
\newcommand{\diam}[0]{\mathtt{diam}}
\newcommand{\diag}[0]{\mathtt{diag}}
\newcommand{\ess}[0]{\mathtt{ess}}
\newcommand{\vertices}[0]{\mathtt{vert}}
\newcommand{\cone}[0]{cone}

\newcommand{\qt}{\mathcal{Q}}
\newcommand{\gr}{\mathtt{gr}}
\newcommand{\id}{\mathtt{id}}

\title{Geometric Realizations with Strong Self-Duality \\[0.5em]
\large Part II: Diameter Graphs, Reuleaux Polyhedra, and Thrackles}

\author{Gábor Damásdi\footnote{Supported by the ERC Advanced Grant no. 882971 GeoScape, by the János Bolyai Research Scholarship of the Hungarian Academy of Sciences, the STARTING-25 project no. 152199 and the EXCELLENCE-24 project no. 151504 Combinatorics and Geometry of the NRDI Fund.} \\ \small HUN-REN Alfr\'ed R\'enyi Institute of Mathematics,\\ \small ELTE Eötvös Loránd University, Budapest }

\begin{document}
\maketitle
\begin{abstract}
 In this paper we construct new examples of diameter graphs and Reuleaux polyhedra in $\mathbb{R}^3$, obtaining a full characterization of their combinatorial structure. For a finite set of points $X\subset\mathbb{R}^d$, its diameter graph is the graph on vertex set $X$ where pairs forming a diameter pair are connected by an edge. Grünbaum, Heppes and Straszewicz independently proved that the diameter graph of $X\subset \R^3$ has at most $2|X|-2$ edges, answering a question of Vázsonyi. Their proof relied on ball polytopes. The ball polytope $\mathcal{B}(X)$ is the intersection of the unit balls centered at the points of $X$. 
 
 We call a ball polytope a Reuleaux polyhedron if the centers form a family with $2|X|-2$ diameter pairs. Kupitz, Martini and Perles showed that the skeleton of a Reuleaux polyhedron must be a 2-connected strongly involutive self-dual graph. They conjectured that in the simple 3-connected case this is also sufficient. We not only confirm this conjecture, but we show that any 2-connected (not necessarily simple) strongly involutive self-dual graph arises as the skeleton of a Reuleaux polyhedron. 
 
To construct the new Reuleaux polyhedra we construct new diameter graphs.  It was known that any 3-dimensional diameter graph is a subgraph of a non-bipartite quadrangulation of the projective plane. We show that the reverse holds. That is, for any such graph we construct a diameter realization.  This also confirms and strengthens a conjecture of Montejano, Pauli, Raggi, Roldán-Pensado on metric embeddings of strongly involutive self-dual graphs. The construction relies on ideas from rigidity theory.

We also discuss a  number of applications of these results, such as the construction of bodies of constant width and connections to Steinitz’s theorem and Borsuk's conjecture.
 \end{abstract}

\section{Introduction}

In the first part of this series we studied the correspondence between 2-connected strongly involutive self-dual maps (SISD maps) and  simple non-bipartite quadrangulations of the projective plane (NBQP maps). In this paper, expanding the work of Kupitz, Martini and Perles \cite{kupitz2010ball}, we  show how this correspondence naturally arises in extremal diameter graphs in $\R^3$ and skeletons of the corresponding Reuleaux polyhedra. Furthermore, we show that any non-bipartite quadrangulation map gives rise to an extremal diameter graph and hence any 2-connected strongly involutive self-dual map is the skeleton of a Reuleaux polyhedron. This settles a question of Kupitz-Martini-Perles and confirms a conjecture of Montejano, Pauli, Raggi, Roldán-Pensado.  

\subsection{Summary of Part I}

In Part I of the series we studied simple non-bipartite quadrangulations of the projective plane and strongly involutive self-dual graphs and showed a natural correspondence between them. A simple non-bipartite quadrangulation (NBQP for short) is a 2-cell embedding $Q$ in the projective plane of a simple non-bipartite graph such that each face has size four. The diagonal map $\diag(Q)$ of $Q$ is a planar map, constructed as follows. We lift $Q$ to the sphere and obtain a planar embedding $\widetilde Q$ of the canonical double cover $\gr(Q)\times K_2$ of $Q$. This cover is a planar quadrangulation, hence it is bipartite. Then $\diag(Q)$ is drawn by selecting one of the color classes and drawing the diagonal spanned by the selected color class in each quadrangle. 

If $G=(V,E,F)$ is a self-dual planar map, we say that the  self-duality  $\tau_G:V\cup E\cup F \rightarrow V\cup E\cup F$ is \emph{strongly involutive} if it satisfies the following conditions:
\begin{enumerate}
    \item[1)]  The self-duality is an involution, i.e., $\tau_G^2=\id$.
    \item[2)]  For every vertex $u\in V$ we have $u \notin \tau_G(u)$.
\end{enumerate}
Planar maps that have a strongly involutive self-duality are known as strongly involutive self-dual maps, SISD maps for short.
The diagonal map of an NBQP is always a 2-connected SISD map. 

In the other direction, consider a planar map $G=(V,E,F)$. One can construct its radial map $R(G)$, also known as its vertex-face incidence map, as follows. The vertex set is $V\cup F$. We place each face-vertex in the corresponding face of $G$. The edges are given by face-vertex incidences, that is, we connect each face-vertex to the vertices on the boundary of the corresponding face. The radial map is always a planar quadrangulation. If $G$ is strongly involutive self-dual, then the quadrangulation  has an automorphism given by the self-duality of the original graph. As the self-duality is strongly involutive, this will be a fixed point free automorphism and we can factor the map by this automorphism. We denote the quotient map by $\overline R(G):=R(G)/\overline{\tau_G}$; it is a quadrangulation of the projective plane. There can be only one strongly involutive self-duality, hence this is well defined. The main result of Part I states that these operations are inverse.

\begin{theorem}[SISD--NBQP correspondence]\label{thm:commute} The two operations, $\diag()$ and $\overline R()$, are inverse bijections between the family of simple non-bipartite quadrangulations of the projective plane  and the family of 2-connected strongly involutive self-dual maps. They satisfy the following properties.  
\begin{enumerate}
    \item[a)] If $G$ is a 2-connected strongly involutive self-dual map, then $G=\diag(\overline R(G))$ as planar maps.
    \item[b)] If $Q$ is a non-bipartite quadrangulation of the projective plane, then $Q=\overline R(\diag(Q))$ as planar maps.
    \item[c)] If $Q$ is a non-bipartite quadrangulation of the projective plane with cover map $\widetilde Q$, then $\widetilde Q= R(\diag(Q))$ as projective planar maps.

\end{enumerate}
    \end{theorem}

This gives a very strong correspondence between the properties of SISD maps and NBQP maps, which are summarized in Table \ref{tab:nbqp-sisd-dictionary}.

 \begin{table}[ht]
\centering
\renewcommand{\arraystretch}{1.25}
\begin{tabular}{c|c}
\textbf{NBQP side} & \textbf{SISD side} \\
\hline
Non-bipartite quadrangulation $Q$ of $\Pb^2$
&
Reduced radial map $\overline R(G)$
\\
Diagonal map  $\diag(Q)$
&
2-connected strongly involutive self-dual planar map $G$
\\

Canonical double cover map $\widetilde Q$
&
Radial map $R(G)$
\\

A face quadrangle of $Q$
&
A primal-dual edge pair $e,e^*$ of $G$
\\

The two diagonal classes of $\widetilde Q$
&
The primal and dual edge sets of $G$
\\

No separating $4$-cycle in $Q$
&
3--connectivity of $\diag(Q)$
\\

A separating $4$-cycle in $Q$
&
A $2$-cut in $\diag(Q)$
\\

\end{tabular}
\caption{Dictionary between non-bipartite quadrangulations of the projective plane and 2-connected strongly involutive self-dual planar maps.}
\label{tab:nbqp-sisd-dictionary}
\end{table}

\subsection{New results}
Borsuk conjectured in 1932 that any bounded subset $X$ of the space $\R^d$ can be partitioned into $(d + 1)$ sets, each of which has diameter smaller than that of $X$ \cite{borsuk1933drei}. Although it was shown to be false by Kahn and Kalai \cite{kahn1993counterexample}, this question has remained a central question in discrete geometry ever since. See \cite{Kalai_2015} for a survey of the topic. 

Let $\diam(X)$ denote the diameter of the set $X$. For a finite $X\subset \R^d$, the \emph{diameter graph} $\Gamma(X)$ is the graph on vertex set $X$ where pairs forming a diameter pair are connected. For finite $X$, Borsuk's question is equivalent to determining the chromatic number of diameter graphs in $\R^d$. 

In $\R^2$ diameter graphs are well understood. 
Hopf and Pannwitz showed that a cycle in a diameter graph in the plane is necessarily odd \cite{ERD46,hopf1934aufgabe}.  Sutherland and Fenchel showed that any diameter graph in $\R^2$ has at most $n$ edges, and this is attainable for $n\ge 3$. Their results imply  the following characterization of diameter graphs in $\R^2$. A graph $G$ is a \emph{odd hairy cycle} if it can be obtained by attaching degree 1 vertices to the vertices of an odd cycle.   A graph $G$ is a diameter graph in $\R^2$ if and only if it is a subgraph of an odd hairy cycle and it has at least one edge (see, for example, Figure \ref{fig:2dexample}). 
\begin{figure}[!ht]
    \centering
    \includegraphics[width=0.2\linewidth]{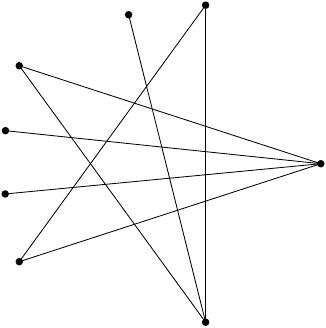}
    \includegraphics[width=0.2\linewidth]{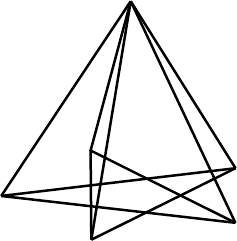}
    \caption{A maximal diameter graph in $\R^2$ and in $\R^3$.}
    \label{fig:2dexample}
\end{figure}

Vázsonyi conjectured that any diameter graph in $\R^3$ has at most $2n-2$ edges \cite{ERD46}. This was shown to be true independently by Grünbaum, Heppes, and Straszewicz using ball polytopes \cite{GRU56,HEP56, STR57}. Alternative proofs were given by Kupitz, Martini, Perles \cite{kupitz2010ball}, Pinchasi \cite{ROM08}, and Swanepoel \cite{sw08}. Dolnikov showed that in a diameter graph in $\R^3$ any two odd cycles have a common vertex \cite{DOL00}.  Diameter graphs are also studied in higher dimensions, see for example \cite{erdHos1959sets,Kupavskii2015,swanepoel2009unit,KupavskiiPolyanskii2017Schur}.

In this paper we focus on the three-dimensional case. We say that $X\subset \R^3$ is an \emph{extremal set} for the Vázsonyi problem if $\diam(X)=1$, $|X|=n$ and the diameter graph has $2n-2$ edges. The most detailed analysis of extremal sets was done in \cite{kupitz2010ball}. They showed that extremal sets are closely connected to a special family of ball polytopes. Let $B(x)$ denote the unit ball centered at $x$. Let $\Bc(X):=\bigcap_{x\in X} B(x)$, that is $\Bc(X)$ is the intersection of the unit balls centered at the points of $X$. If $X$ is finite, we say that $\Bc(X)$ is a \emph{ball polytope}, that is a ball polytope is the intersection of finitely many unit balls.  Ball polytopes play a key part in several important problems in discrete geometry, see \cite{bezdek2007ball,sallee1970reuleaux,bezdek2002pushing,bezdek2026selected}. If $X$ is extremal, then we say that $\Bc(X)$ is an \emph{Reuleaux polyhedron}. We must emphasize that there are several different notions in the literature under this name. Here we decided to follow the convention of Hynd \cite{hynd2024density}. This is somewhat different from the choice of  Montejano and Roldán-Pensado \cite{MON17} and Sallee \cite{sallee1970reuleaux}, who require various additional conditions from $X$, not just extremality.

The boundary of the ball polytope $\Bc(X)$ is similar to the boundary of a polytope: it has spherical facets, circular edges and vertices. (We also allow dangling vertices; see Section \ref{sec:cellstructure} for details.) The skeleton $\skel\Bc(X)$ of $\Bc(X)$ is the map formed by the vertices and edges of the boundary with its natural planar embedding. The skeleton is always planar, but, contrary to the usual convex 3-polytopes, it is not always simple or 3-connected, and the face structure is not necessarily a lattice.  The skeleton of a ball polytope is always 2-connected; furthermore, if $X$ is extremal for the Vázsonyi problem, then $\skel \Bc(X)$ is a strongly involutive self-dual graph \cite{kupitz2010ball}.

Our first contribution is that we show that the correspondence from Part I naturally arises in the setting of extremal sets and Reuleaux polyhedra. We show that if $X$ is extremal, then the diameter graph $\Gamma(X)$ has a natural embedding $\qt(X)$ in the projective plane as a non-bipartite quadrangulation. We will call $\qt(X)$ the \emph{diameter quadrangulation of $X$}. See Figure \ref{fig:diamquadex} and Appendix \ref{app:examples} for examples. This idea is implicit in the literature, see \cite{kupitz2010ball}, but it will be very useful to formulate it explicitly. 

\begin{figure}[!ht]
    \centering
    \includegraphics[width=0.2\linewidth]{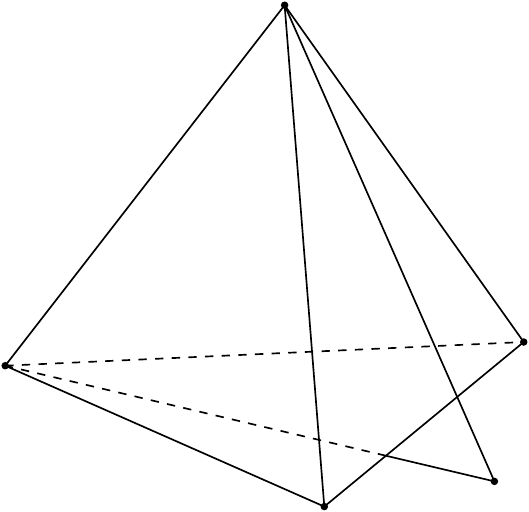}
    \includegraphics[width=0.2\linewidth]{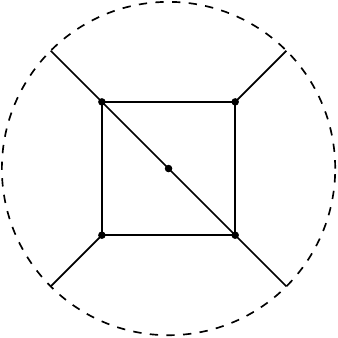}
    \caption{An extremal set $X$ in $\R^3$ and the corresponding quadrangulation $\qt(X)$.}
    \label{fig:diamquadex}
\end{figure}

We define $\qt(X)$ as the reduced radial map $\overline R(\skel\Bc(X))$. Since $\skel\Bc(X)$ is 2-connected and strongly involutive self-dual, this is indeed a quadrangulation. Then, in Section \ref{sec:cellstructure}, we show that the graph of $\qt(X)$ is  the diameter graph of $X$. A good image to keep in mind is that a diameter pair $(x,y)$ encodes that the face $F_x$ of $\Bc(X)$ that is centered at $x$ is incident to the vertex $y$ and equivalently the face $F_y$ is incident to $x$. Hence, the diameter graph is the graph of the reduced radial map of the skeleton of $\Bc(X)$. Furthermore, the embedding of the reduced radial map can also be found in the geometric setting in the following way. Consider a point $p\in X$ in an extremal set. The neighbors of $p$ in $\Gamma(X)$ span a pointed cone with apex at $p$; we call this the neighborhood cone. Each neighbor of $p$ determines an extremal vector of this cone. We will see that the order of the neighbors of $p$ in $\qt(X)$  matches the order of the extremal vectors around the neighborhood cone at $p$. (See Lemma \ref{lem:diameterispolyhedral} for details).

The definition of $\qt(X)$ and Theorem \ref{thm:commute} immediately implies the following.

\begin{corollary}\label{cor:skeldiagonbody}
    If $X$ is an extremal set of $n\ge 4$ distinct points, then the  skeleton of the corresponding Reuleaux polyhedron is the same as the diagonal map of the corresponding diameter quadrangulation. In other words, $\skel(\Bc(X))=\diag(\qt(X))$ as planar maps.
\end{corollary}

This implies the correspondence in Table \ref{tab:extremal-ball polytope-dictionary}.

\begin{table}[ht]
\centering

\renewcommand{\arraystretch}{1.25}
\begin{tabular}{c|c}
\textbf{Extremal set / diameter side}
&
\textbf{Reuleaux polyhedra }
\\
\hline

Extremal point set $X\subset \mathbb R^3$
&
Ball polytope $\Bc(X)$
\\

Diameter pair $pq$
&
Vertex--facet incidences $p\in F_q \Leftrightarrow  q\in F_p$
\\

Diameter quadrangulation $\qt(X)$
&
Reduced radial map $\overline R(\skel(\Bc(X)))$
\\

Face quadrangle of $\qt(X)$
&
Dual pair of edges in $\skel(\Bc(X))$
\\
$\diag(\qt(X))$
&
Skeleton $\skel(\Bc(X))$
\\

\end{tabular}
\caption{Dictionary between extremal diameter configurations and their associated ball polytopes.}
\label{tab:extremal-ball polytope-dictionary}
\end{table}

Secondly, we give a full characterization of skeletons of Reuleaux polyhedra. Kupitz et al. formulated the following natural question (Question 2 in \cite{kupitz2010ball}). Which self-dual simple 3-connected planar graphs are combinatorially equivalent to the cell structure of $\Bc(X)$ for some extremal $X$? They conjectured the following.

\begin{conjecture}[Conjecture 9.1 in \cite{kupitz2010ball}]\label{conj:kupitz}
A necessary and sufficient condition for the situation that the face structure of a $3$-polytope $P$ be isomorphic to the face structure of $\Bc(X)$ for some extremal configuration $X \subset \mathbb{R}^3$ is that $P$ is strongly involutive self-dual.
\end{conjecture}

 We confirm their conjecture in a much stronger form\footnote{They mention that there will be a subsequent paper where they construct new families of Reuleaux polyhedra, and they also remark that there is a connection to negatively self-polar polytopes and give some ideas for generating these. As far as the author of this paper knows, the announced paper never appeared. We also note that we are following a different path for solving the conjecture.  
}. We show  that strongly involutive self-duality is also sufficient in the non-simple and only 2-connected setting.

\begin{theorem}\label{thm:main_skeleton}
For every planar (not necessarily simple) 2-connected strongly involutive self-dual map $G$ there is an $X\subset \R^3$ such that $X$ is extremal for the Vázsonyi problem and the skeleton of the Reuleaux polyhedron $\Bc(X)$  is isomorphic to $G$ as a map. 
\end{theorem}

As the skeleton of a Reuleaux polyhedron is always 2-connected and strongly involutive self-dual, Theorem \ref{thm:main_skeleton} gives a full characterization. 

Our proof of Theorem \ref{thm:main_skeleton}  follows a plan suggested by Montejano, Pauli, Raggi and  Roldán-Pensado \cite{montejano2020graphs}. They noted that if $G$ is a strongly involutive self-dual map which is the skeleton of a Reuleaux polyhedron $\Bc(X)$, then we can use $G$ to calculate the diameter graph $\Gamma(X)$ without actually knowing $X$ itself. Hence, if $G$ is given and we are looking for a Reuleaux polyhedron with skeleton $G$, then we can first calculate the desired diameter graph and look for a set $X$ with that diameter graph. They formulated the idea as follows. Let $G=(V,E)$ be a strongly involutive self-dual map with self-duality $\tau_G$.  A \emph{metric mapping} of $G$ is a vector labeling $\vu : V \rightarrow \mathbb{R}^3$ such that the diameter of $\vu(V)$ is 1, and for every $v \in V$ and $u \in \tau_G(v)$ we have $|\vu(u)- \vu(v)| = 1$. If $\vu$ is injective, then it is called a \emph{metric embedding}. In other words, we want to realize the reduced radial map $\overline R(G)$ as a diameter graph in $\R^3$. Montejano, Pauli, Raggi and  Roldán-Pensado studied metric mappings of strongly involutive graphs \cite{montejano2020graphs} and showed that given a strongly involutive self-dual graph $G$, there exists a (not necessarily injective) metric mapping of $G$.  
They conjectured that if $G$ is a polyhedral graph, i.e. 3-connected and simple, then there is always a metric embedding. 

\begin{conjecture}[Montejano, Pauli, Raggi, Roldán-Pensado \cite{montejano2020graphs}]\label{conj:montereuleaux}
 For every simple planar 3-connected strongly involutive self-dual graph $G$ there is a metric embedding.
\end{conjecture}

Of course, their aim was to use this for solving Conjecture \ref{conj:kupitz} by showing that if we have a metric embedding $\vu$ of a simple 3-connected SISD map $G$, then the boundary complex of  $\Bc(\vu(V))$ is isomorphic to the face structure of $G$\footnote{They say that this follows from \cite{kupitz2010ball}. Although this is true, it is highly non-trivial and fails in the non-3-connected case. }. This implies that Conjecture \ref{conj:kupitz} and Conjecture \ref{conj:montereuleaux} are equivalent. To support their conjecture, they enumerated all simple 3-connected strongly involutive self-dual graphs up to 14 vertices and numerically computed metric embeddings for them.

For a given $G$ the metric mapping was found in the following way. First they showed that $\overline R(G)$ is 4-colorable. Then we can simply map the vertices to the corners of a regular tetrahedron using the colors. They noted that we can also view this as a rigidity theoretic framework, we have vertices in $\R^3$ and certain pairs have a fixed distance 1. They argued that the resulting framework is probably not rigid for $n>4$ as we only have $2n-2< 3n-6$ edges. Our proof of Conjecture \ref{conj:montereuleaux} was inspired by this idea. We show that all possible $\overline R(G)$ can be generated starting from a fixed quadrangulation $Q_4$ and applying a vertex splitting operation. Then we show that the same vertex splitting operation can be realized geometrically in diameter graphs using rigidity theory, see Figure \ref{fig:idea}.  

\begin{figure}[!ht]
    \centering
        \includegraphics[width=0.2\linewidth]{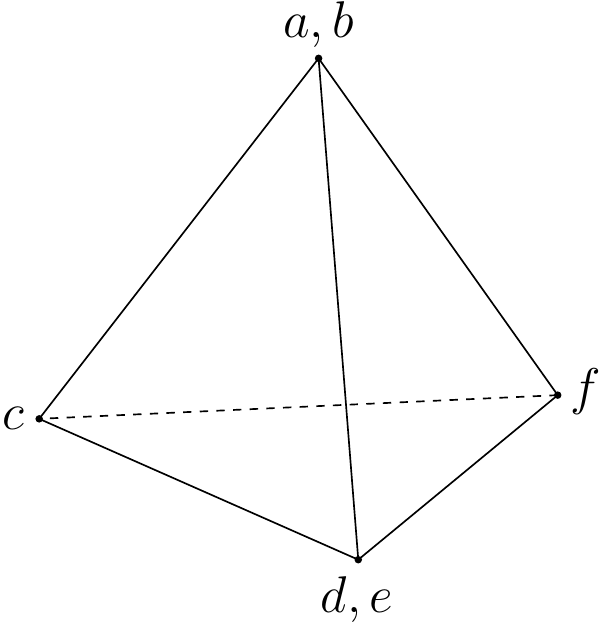}
    \includegraphics[width=0.2\linewidth]{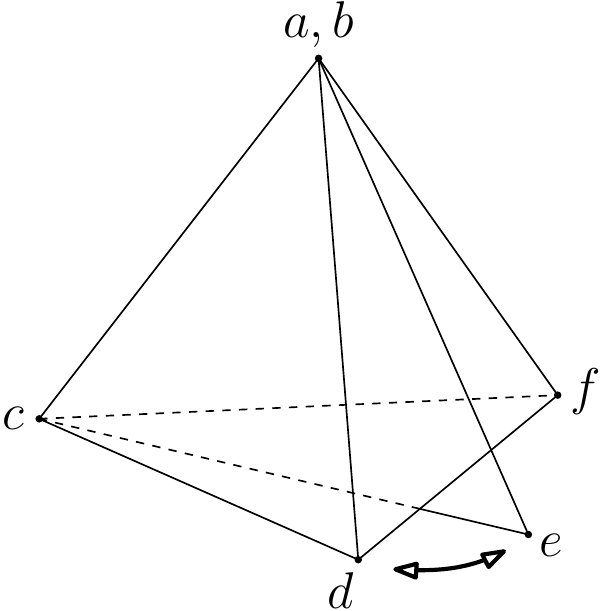}
    \includegraphics[width=0.2\linewidth]{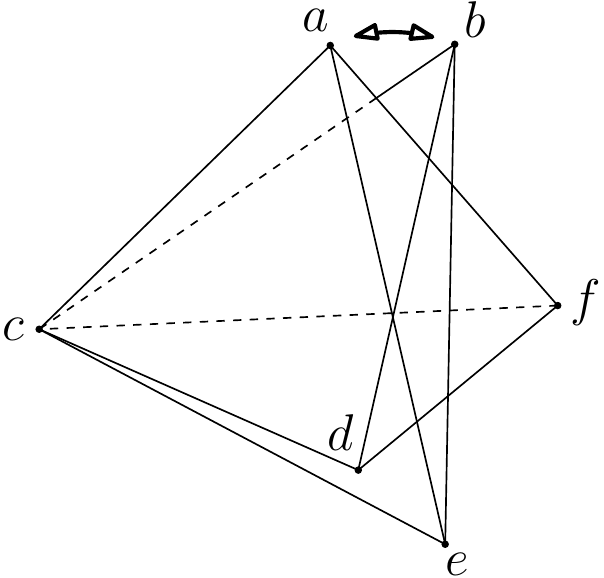}
    \caption{The geometric vertex splitting procedure using rigid frameworks.}
    \label{fig:idea}
\end{figure}

As we will work in the general 2-connected setting, we will actually require something more than a metric embedding of $\overline R(G)$; we want an extremal set $X$ such that its diameter quadrangulation $\qt(X)$ is isomorphic to $\overline R(G)$ as a map\footnote{The reason for this  comes from the following phenomenon. If $G$ is a 2-connected SISD map then $\overline R(G)$ is an NBQP map. But if $G$ is not 3-connected, then the abstract graph of $\overline R(G)$ might have other NBQP embeddings. The metric embedding condition only considers the graph of $\overline R(G)$, which is not enough to apply Theorem \ref{thm:main_skeleton}.}. We show the following strengthening of Conjecture \ref{conj:montereuleaux}.

\begin{theorem}\label{thm:main_nbqp_to_diam}
    For every simple non-bipartite quadrangulation of the projective plane $Q$ there is an extremal set $X\subset \R^3$ such that the diameter quadrangulation $\qt(X)$ is $Q$. 
\end{theorem}

With these tools in hand the proof of Theorem \ref{thm:main_skeleton} is going to be straightforward. Given a 2-connected SISD map $G$, the idea is to take the corresponding quadrangulation $Q$ given by Theorem \ref{thm:commute}. Then we realize $Q$ as a diameter quadrangulation $\qt(X)$ (with correct rotations) using Theorem \ref{thm:main_nbqp_to_diam}. Finally, using Corollary \ref{cor:skeldiagonbody}, we argue that the extremal set $X$ gives a Reuleaux polyhedron $\Bc(X)$  with skeleton $G$. See Figure \ref{fig:mainsteps} for an illustration.

\begin{figure}[!ht]
    \centering
    \begin{subfigure}[b]{0.3\linewidth}
        \centering
        \includegraphics[width=\linewidth]{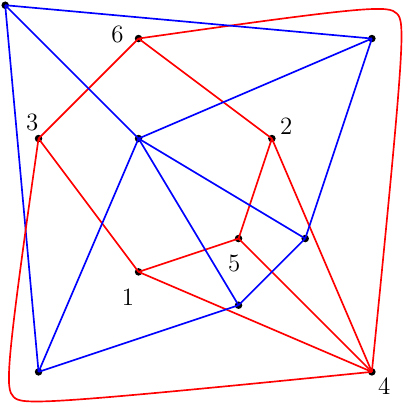}
        \caption{A strongly involutive self-dual map $G$.}
        \label{fig:mainsteps-a}
    \end{subfigure} 
    \hspace{15px}
    \begin{subfigure}[b]{0.3\linewidth}
        \centering
        \includegraphics[width=\linewidth]{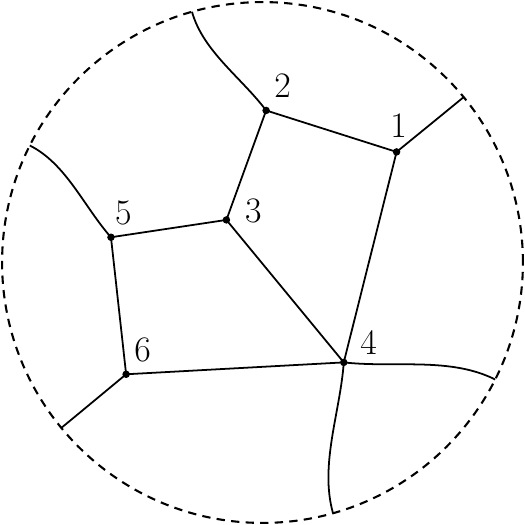}
        \caption{The corresponding quadrangulation $Q=\overline{R}(G)$.}
        \label{fig:mainsteps-b}
    \end{subfigure}
    
    \begin{subfigure}[b]{0.3\linewidth}
        \centering
        \includegraphics[width=\linewidth]{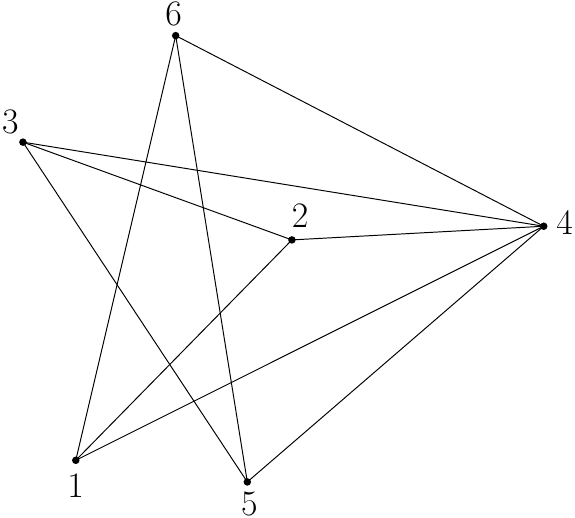}
        \caption{An extremal set $X$ so that $Q=\qt(X)$.}
        \label{fig:mainsteps-c}
    \end{subfigure}
    \hspace{15px}
    \begin{subfigure}[b]{0.3\linewidth}
        \centering
        \includegraphics[width=\linewidth]{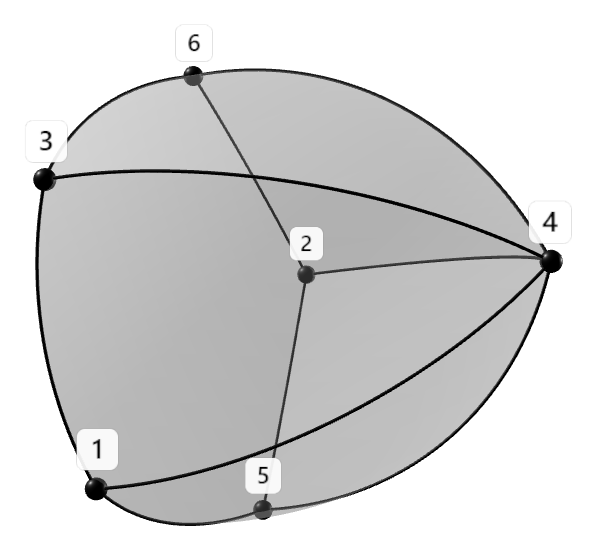}
        \caption{The Reuleaux polyhedron $\Bc(X)$ whose skeleton is $G$.}
        \label{fig:mainsteps-d}
    \end{subfigure}
    \caption{Steps of the proof of Theorem \ref{thm:main_skeleton}.}
    \label{fig:mainsteps}
\end{figure}

We also show a surprising application of our results. A \emph{topological graph} is a graph drawn in the plane so that each edge is represented by a Jordan arc, any two edges of the graph intersect finitely many times, no edge passes through a vertex different from its endpoints, and no two edges touch. A  topological graph is a \emph{thrackle} if  every pair of edges meet exactly once, where a common endpoint is also considered as a meeting point. A graph that has a thrackle drawing is called \emph{thrackleable}. A famous conjecture of Conway says that any thrackle on $n$ vertices has at most $n$ edges. After a number of subsequent improvements \cite{lovasz1997conway,CN00,fulek2011computational} the best known upper bound is $1.393n$ by Xu \cite{xu2021new}. 

Lovász et al. studied the notion of \emph{generalized thrackles}. These are topological graphs in which every pair of edges meet an odd number of times, where common endpoints are also counted \cite{lovasz1997conway}. They showed that every generalized thrackleable graph has at most $3n-4$ edges, which was improved by Cairns and Nikolayevsky to the sharp upper bound $2n-2$ \cite{CN00}. The connection to diameter graphs came from Perlstein and Pinchasi.

\begin{theorem}[Perlstein, Pinchasi \cite{ROM08}]\label{thm:diamisgent}
    If $X$ is a finite set of points in $\mathbb{R}^3$, then the diameter graph $\Gamma(X)$ is generalized thrackleable.
\end{theorem}

It turns out that a graph is generalized thrackleable if and only if it is the subgraph of an NBQP graph. Using this idea, we will show the following equivalence. 

\begin{theorem}[Diameter graph characterization]\label{thm:main_thrackleable}
    Let $G$ be a simple graph with at least one edge. Then the following are equivalent:
    \begin{enumerate}
        \item[a)] $G$ is a diameter graph in $\R^3$.
        \item[b)] $G$ is generalized thrackleable.
        \item[c)] $G$ is a subgraph of a simple non-bipartite quadrangulation of the projective plane.
    \end{enumerate}
     
\end{theorem}

The paper is structured as follows. In Section \ref{sec:rotation} we show how to define rotation systems for NBQP maps. In Section \ref{sec:cellstructure} we recall the results of \cite{kupitz2010ball} on the cell structure of Reuleaux polyhedra and their connection to diameter graphs. In Section \ref{sec:quadofext} we introduce and study $\qt(X)$, the diameter quadrangulation of an extremal set $X$.   In Section \ref{sec:generating} we recall a generating method for NBQP maps due to Negami and Nakamoto. We show how it implies a generating method for SISD maps. In Section \ref{sec:rigidity} we introduce the rigidity theoretic tools and show that frameworks given by extremal diameter graphs in $\R^3$ are always stress-free. In Section \ref{sec:qtexsists} we show our main results, Theorem \ref{thm:main_nbqp_to_diam} and Theorem \ref{thm:main_skeleton}. In Section \ref{sec:thrackles} we discuss the connection to thrackles and prove Theorem \ref{thm:main_thrackleable}. In Section \ref{sec:furtherconnections} and \ref{sec:openproblems} we discuss open problems and further applications. Appendix \ref{app:examples} shows a collection of examples.

\section{Preliminaries}

All terminology follows Part~I. In particular, isomorphisms, automorphisms, and dualities are map-theoretic unless explicitly stated otherwise. We use $M$ for general maps, which are equivalence classes of topologically equivalent embeddings.  We use $\gr(M)$ to denote the abstract graph of the map $M$. A graph (or map) is \emph{simple} if it does not contain loop edges or parallel edges. We do not assume that our graphs are simple.

We use $G$ for SISD maps. The unique strongly involutive self-duality of a SISD map $G$ is denoted by $\tau_G$.  We write $R(G)$ for the radial map and $\overline R(G):=R(G)/\overline{\tau_G}$ for the reduced radial map, where $\overline{\tau_G}$ is the map automorphism of $R(G)$ induced by $\tau_G$. 

 NBQP maps are denoted by $Q$; these are always simple graphs by definition. The canonical cover of an NBQP map $Q$ is denoted by $\widetilde Q$, which is a map of the canonical double cover $\gr(Q)\times K_2$. The diagonal map of $Q$ is denoted by $\diag(Q)$.

Finite point sets in $\R^3$ are denoted by $X$.  The diameter of a set is denoted by $\diam(X)$, while the diameter graph is denoted by $\Gamma(X)$.  The ball polytope generated by $X$ is $\Bc(X)$, and the NBQP map associated with an extremal set $X$ is denoted by $\qt(X)$.

Let $\Kc \subset \R^d$ be a closed convex cone. We say that $\Kc$ is \emph{pointed} if $\Kc\cap -\Kc=\{\mathbf{0}\}$. 
A cone $\Kc$ is \emph{polyhedral} if it can be written as a set of solutions of finitely many linear equalities and inequalities. The \emph{face} of a polyhedral cone $\Kc$ is a convex cone $F\subset \Kc$ that can be written as $F= \Kc\cap H$ for a supporting hyperplane $H$ of $\Kc$. An \emph{extreme ray} is a 1-dimensional face. A 3-dimensional pointed polyhedral cone has finitely many extreme rays. The \emph{rotation} of a pointed cone is the cyclic order of the extreme rays, that is, we list the extreme rays from the viewpoint of the apex in counterclockwise order; see Figure \ref{fig:conerotdef}.

\begin{figure}[!h]
    \centering
    \includegraphics[width=0.15\linewidth]{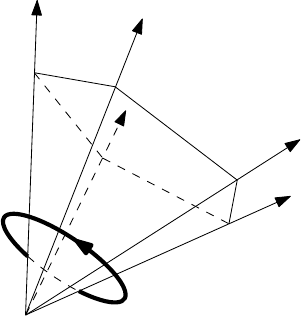}
    \caption{The rotation of a 3-dimensional pointed polyhedral cone.}
    \label{fig:conerotdef}
\end{figure}
Note that in the definition of the rotation we do not assume that the cone is full dimensional. In particular, the cone might have only 2 extreme rays $\mathbf{a}$ and $\mathbf{b}$. In this case the rotation is simply $(\mathbf{a},\mathbf{b})$.

For $\vc\in \R^d$ and $S\subset \R^d$ let $\cone_\vc(S)$ denote the translated cone that is generated by $\{\vv-\vc\mid \vv\in S\}$ and translated by $\vc$. We call $\vc$ the \emph{apex} of this translated cone, and use the terms pointed, polyhedral, face, extreme ray and rotation for the corresponding notions after translating the apex to the origin.

\section{Rotation systems for NBQP maps}\label{sec:rotation}

One way to describe the combinatorial structure of a simple connected planar map is by listing its vertices, edges and faces, and specifying the incidences between them. Another option is to use rotation systems. For a planar embedding, the \emph{rotation} of a vertex $v$ is the order of the edges around $v$ in counterclockwise order in the embedding. The \emph{rotation system} of a planar embedding $G$ is the collection of the rotations of its vertices. For isomorphic embeddings, the rotation system is either the same or each vertex rotation is reversed, depending on whether the homeomorphism is orientation-preserving or not. Hence a planar map $G$ determines the rotation system up to reversing all vertex rotations. If $G$ is connected, the converse also holds: the rotation system determines the combinatorial structure.

We will need a similar notion for graphs embedded in the projective plane. Unfortunately, the generalization is not straightforward, as the projective plane is not an orientable surface. The general idea is to fix one of the two cyclic orderings around each vertex and fix a signature for each edge, which describes whether following the curve of the edge flips the local orientation or not. See, for example, \cite{gross2001topological} for details. We will only need rotation systems for NBQPs, which simplifies the setting: we can assume that every edge has negative signature, i.e. they are all locally orientation-changing. We will only describe this simplified variant that is tailored for our needs.

Let $Q$ be an NBQP map. If we lift $Q$ to the sphere, we get a  planar map  $\widetilde Q$ of the canonical double cover $\gr(Q) \times K_2$ of $Q$ (Section 4 in Part I). Let $V_1,V_2$ be the color classes of $\gr(Q) \times K_2$ and suppose that each vertex $i$ of $Q$ is covered by $i_1\in V_1$ and by $i_2\in V_2$.  The \emph{rotation} of a vertex $i$ in $Q$ is defined based on the rotation of $i_1$. If the rotation of $i_1$ in $\widetilde Q$ is $[a_2,b_2,c_2,\dots]$, then we define the rotation of $i$ as $[a,b,c,\dots]$, see Figure \ref{fig:rotation} for an example. The collection of the vertex rotations is a \emph{rotation system of $Q$}. 
This is well-defined up to selecting the vertex classes $V_1$ and $V_2$, i.e. up to reversing all the vertex rotations. Hence, we have exactly two rotation systems of $Q$. Contrary to the planar case there is no natural way to pick one of these two rotation systems. 

\begin{figure}[!ht]
    \centering
  
    \begin{subfigure}[t]{0.25\textwidth}
    \centering
    \includegraphics[width=0.99\linewidth]{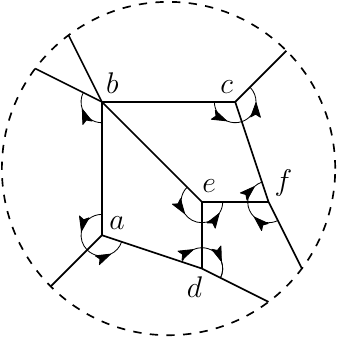}
    
    \label{fig:rotationa}
  \end{subfigure}
  \hspace{10px}
    \begin{subfigure}[b]{0.2\textwidth}
    \centering
    \includegraphics[width=0.99\linewidth]{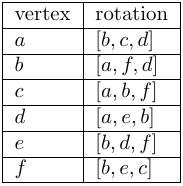}
    
    \label{fig:rotationb}
  \end{subfigure}
  \hspace{10px}
    \begin{subfigure}[t]{0.3\textwidth}
    \centering
    \includegraphics[width=0.99\linewidth]{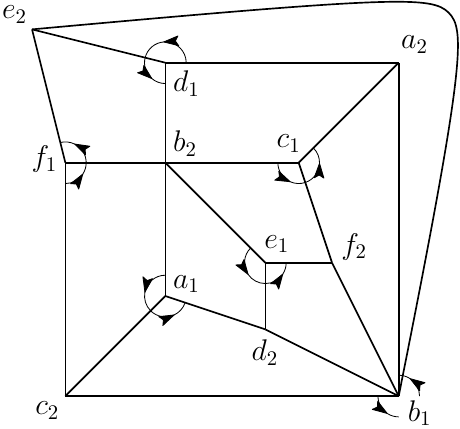}
    
    \label{fig:rotationc}
  \end{subfigure}
    \caption{Rotation in a non-bipartite quadrangulation of $\Pb^2$.}
    \label{fig:rotation}
\end{figure}

Suppose $Q_1$ and $Q_2$ are NBQP maps such that their graphs are isomorphic and the isomorphism also preserves the vertex rotations. Then $Q_1$ and $Q_2$ are the same NBQP map. Indeed, if they have matching rotation systems, then $\widetilde Q_1$ and $\widetilde Q_2$ have matching planar rotation systems. Hence $\widetilde Q_1$ and $\widetilde Q_2$ are equivalent planar maps and therefore so are $Q_1$ and $Q_2$.

Also note that the rotation system of $Q$ naturally interlaces with the rotation system of the diagonal graph $\diag(Q)$.

\section{Cell structure of Reuleaux polyhedra}\label{sec:cellstructure}
The purpose of this section is to recall some results on the structure of ball polytopes that are needed later. We follow the terminology of Kupitz--Martini--Perles~\cite{kupitz2010ball}. The structure of ball polytopes was also studied in \cite{bezdek2007ball,martini2019bodies,almohammad2022analogue,lopez2023borsuk}. 

Throughout, $X\subset \R^3$ is a finite set.  For $p\in\R^3$, let $B(p):=\{x\in \R^3:\|x-p\|\leq 1\}$ be the unit ball, and let $S(p):=\bd B(p)=\{x\in \R^3:\|x-p\|=1\}$ be the unit sphere centered at $p$. The circumradius of a bounded set \(W\) is denoted by \(cr(W)\). For a finite set \(X\subset \R^3\) with $cr(X)<1$, its associated ball polytope is
$$
        \Bc(X):=\bigcap_{x\in X} B(x).
$$

As $cr(X)<1$, $\Bc(X)$ is full-dimensional. The ball polytope $\Bc(X)$ is not a polytope in the ordinary sense: its two-dimensional faces lie on unit spheres, its edges are circular arcs, and its face structure may not be isomorphic to the face structure of a convex 3-polytope.

A point $p\in X$ is \emph{essential} if removing the unit ball centered at $p$ strictly changes the ball polytope: $        \Bc(X)\subsetneq \Bc(X\setminus\{p\})$. Equivalently, $p$ is \emph{inessential} if $\Bc(X)=\Bc(X\setminus\{p\})$. Let $\ess(X):=\{p\in X:p\text{ is essential}\}$. Then  $\Bc(\ess(X))=\Bc(X)$. A finite set $X\subset\R^3$ is called \emph{tight} if $cr(X)<1$ and $X=\ess(X)$. That is, a tight set has no redundant ball centers.

\begin{theorem}[Theorem 5.2 from \cite{kupitz2010ball}]
Let \(X\subset\R^3\) be finite and assume \(\diam(X)=1\).  Then:
\begin{enumerate}[label=\textup{(\roman*)}]
    \item \(cr(X)<1\).
    \item If \(p\in X\) is incident with at least two diameters of \(X\),
    then \(p\) is essential.
    \item If \(X\) is extremal for the Vázsonyi problem, then \(X\) is tight.
\end{enumerate}
\end{theorem}

Assume now that \(X\) is tight and \(cr(X)<1\).  For a point $p \in X$ let $F_p = \{x \in
\Bc(X) : |x - p| = 1\}$ denote the \emph{facet} of $\Bc(X)$ corresponding to $p$. For tight $X$, the facets partition the boundary as one would expect: their union is the whole boundary, but their relative interiors are pairwise disjoint. On the other hand, $F_p\cap F_q$ might be two vertices of $X$, which cannot happen for convex 3-polytopes. See the examples \ref{fig:example-B}, \ref{fig:example-E} and \ref{fig:example-G}.

 A point $x\in \bd(\Bc(X))$ is a \emph{vertex} of $\Bc(X)$ if it belongs to at least three facets of $\Bc(X)$ or if it is in $X\cap \Bc(X)$ and belongs to exactly two facets. In the first case the vertex is called a \emph{principal vertex}, in the second case it is called a \emph{dangling} vertex. The set of vertices of $\Bc(X)$ is denoted by $\vertices(\Bc(X))$. Examples \ref{fig:example-B} and \ref{fig:example-E} have dangling vertices. 

  An edge of $\Bc(X)$ is the closure of a connected component of $(F_p\cap F_q) \setminus (\vertices(\Bc(X)))$, where $\{p, q\}$ is a pair of distinct points of $X$. The need for such a definition comes from the fact that $F_p\cap F_q$ might not be connected and might contain dangling vertices. We can have parallel edges but no loop edges.

 The graph formed by the vertices and edges of $\Bc(X)$ is called the skeleton of $\Bc(X)$ and denoted by $\skel \Bc(X)$. We understand $\skel \Bc(X)$ as a planar map using its natural embedding on $\bd (\Bc(X))$. The skeleton is not necessarily simple, see Examples \ref{fig:example-B} and \ref{fig:example-E} that have parallel edges. Even if the minimum degree is at least 3, the skeleton is not necessarily 3-connected, see Example \ref{fig:example-G}.

\subsection{Structure of the boundary}
The skeleton of a ball polytope does not need to be 3-connected, but it is 2-connected.

\begin{theorem}[Theorem 6.1 in \cite{kupitz2010ball}]\label{thm:2connect}
For a finite $X\subset \R^3$, which is tight and $|X|\ge3$, the map \(\skel(\Bc(X))\) is 2-connected.
\end{theorem}

The following result captures the deep connection between the Vázsonyi problem and the corresponding ball polytope. The set $X$ is extremal if and only if the set of vertices of $\Bc(X)$ coincides with the set $X$ itself. 

\begin{theorem}[Extended Grünbaum--Heppes--Straszewicz theorem, Theorem 7.1 in \cite{kupitz2010ball}]\label{thm:extendedghs}
Let \(X\subset\R^3\) be finite with \(|X|=n\geq 4\) and \(\diam(X)=1\).  The
following are equivalent:
\begin{enumerate}[label=\textup{(\roman*)}]
    \item \(X\) is extremal for the Vázsonyi problem.
    \item $X$ has $2n-2$ diameters.
    \item \(X\) is tight and $X=\vertices \Bc(X)$.
    
\end{enumerate}
\end{theorem}

The connection between the diameter graph $\Gamma(X)$ and $\skel(\Bc(X))$ is given by the following lemma. 

\begin{lemma}[Combination of Lemma 6.1 and 6.4 of \cite{kupitz2010ball} for the extremal case]\label{lem:precise}
    Let $X$ be an extremal set of points and for a vertex $c$ let $\{v \in X : |v - c| = 1\} = \{v_1, v_2,\dots , v_m\}$. Then we have the following. 
    \begin{enumerate}
        \item[a)]  The polyhedral cone $\Kc_c=\cone_c(v_1, \dots , v_m)$ is pointed with apex $c$, and the rays $\cone_c(v_i), i = 1, 2, \dots , m$, are the extreme rays of it.
        \item[b)] If they appear in this order, then the facets of $\Bc(X)$ that meet at $c$ are precisely     \(F_{v_1},\ldots,F_{v_m}\).
        \item[c)] These facets are separated by edges $e_i$ with $e_i \subset F_{v_i}\cap F_{v_{i+1}} , i = 1, \dots , m$, hence their order matches the order of $v_1,\dots, v_m$.
    \end{enumerate}

\end{lemma}

For a planar map $M$ let $B(M)$ denote its barycentric subdivision. The skeleton $\skel(\Bc(X))$ has an involutive self-duality that is fixed point free when acting on the barycentric subdivision.

\begin{theorem}[\cite{kupitz2010ball}]\label{thm:canonicalselfduality}
    Suppose $G=(V,E,F)$ is the skeleton of $\Bc(X)$, where $|X|\ge 4$ and $X$ is extremal. Then $G$ has a canonical self-duality $\tau_X:V\cup E\cup F\rightarrow V\cup E\cup F$ with the following properties:
    \begin{enumerate}
        \item $\tau_X$ is involutive, that is, $\tau_X^2=\id$.
        \item For every $p\in X$ we have $\tau_X(p)=F_p$.
        \item For every edge $uv\in E$ we have that the endpoints of $\tau_X(uv)$ are included in both facets $F_u$ and $F_v$.
        \item The map automorphism of $B(G)$ that is induced by $\tau_X$ is fixed point free. 
    \end{enumerate}
\end{theorem}

This immediately implies that the skeleton is a strongly involutive self-dual map with the canonical self-duality.

\begin{corollary}\label{cor:canonicalissisd}
    Suppose $X$ is extremal and $|X|\ge 4$. Then the canonical self-duality of $\skel\Bc(X)$ is a strongly involutive self-duality.
\end{corollary}
\begin{proof}
    The self-duality is involutive by Theorem \ref{thm:canonicalselfduality}. As we have $\tau_X(p)=F_p$, it is not possible that $p\in\tau_X(p)$. 
\end{proof}

\section{Quadrangulation corresponding to an extremal set}\label{sec:quadofext}
A \emph{vector labeling} of a map  $Q=(V,E,F)$  is a function $\vu:V\rightarrow \R^d$. That is, we specify the position of the vertices of $Q$ in $\R^d$. A pair $(Q,\vu)$ is called a \emph{framework} if $\vu(p)\neq \vu(q)$ for any edge $(p,q)\in E$.  We say that a framework $(Q,\vu)$ is \emph{neighborhood-coned} if $\Kc_i=\cone_{\vu_i}(\{\vu_j:j\in N(i)\}) $ is a pointed cone for each $i\in V$ and for each $ij\in E$ the cone $\cone_{\vu_i}(\vu_j)$ is an extreme ray of $\Kc_i$. The $\Kc_i$ cones will be called the \emph{neighborhood-cones} of $(Q,\vu)$. For example, the framework given by the skeleton of a convex 3-polytope is neighborhood-coned.

Recall that the \emph{rotation} of a pointed cone is the cyclic order of the extreme rays, that is, we list the extreme rays from the viewpoint of the apex in counterclockwise order. The \emph{cone rotation system} of a neighborhood-coned framework is the collection of the cone rotations of the $\Kc_i$ cones. For example, the cone rotation system of the framework given by the skeleton of a polytope is nothing else but the rotation system of the skeleton viewed as a planar graph embedded on the surface of the polytope.

Now we show that extremal diameter graphs are also neighborhood-coned. At the same time we show that for each extremal $X$ the diameter graph $\Gamma(X)$ has a natural embedding in  $\Pb^2$ as a simple non-bipartite quadrangulation $\qt(X)$. Recall that we define $\qt(X)$ as $\overline R(\skel\Bc(X))$. For an extremal set $X$ we will say that $\qt(X)$ is the \emph{diameter quadrangulation of $X$}.  In the other direction, if $Q$ is a given NBQP, then an $X$ for which $Q=\qt(X)$ is called a \emph{diameter realization of $Q$}. By Corollary \ref{cor:canonicalissisd} and Theorem \ref{thm:2connect} the map $\skel\Bc(X)$ is 2-connected and strongly involutive self-dual. Hence by Theorem \ref{thm:commute} $\qt(X)$ is indeed a well defined simple non-bipartite quadrangulation of the projective plane. We show that the graph of $\qt(X)$ is the diameter graph $\Gamma(X)$ and that the vertex rotations of $\qt(X)$ correspond to the orders of the extreme rays of the neighborhood cones. 

\begin{lemma}\label{lem:diameterispolyhedral}
    If $X$ is an extremal point set consisting of $n\ge 4$ distinct points, then there is a unique simple non-bipartite quadrangulation $\qt(X)=(V,E)$ of the projective plane such that there is a bijective vector labeling $\vu:V\rightarrow X$ with the following properties. 
    \begin{enumerate}[label=\alph*)]
    \item $ij\in E$ if and only if $\vu_i\vu_j$ is a diameter of $X$.
    \item $(\qt(X),\vu)$ is neighborhood-coned.
   \item The cone rotation system of $(\qt(X),\vu)$ matches one of the rotation systems of $\qt(X)$.
    \end{enumerate}
Furthermore, this unique quadrangulation is $\overline{R}(\skel\Bc(X))$.
    
\end{lemma}

\begin{proof}
    First we argue for uniqueness. If there is such a quadrangulation $\qt(X)$, then its graph is fixed by $a)$. Its two rotation systems are fixed by $b)$ and $c)$. As the graph and the rotation system determine the NBQP map $\qt(X)$ is determined by the properties $a),b),c)$.
    
    Next, we show that $\overline{R}(\skel\Bc(X))$ works. By Corollary  \ref{cor:canonicalissisd} we know that $\skel\Bc(X)$ is strongly involutive self-dual and by Theorem \ref{thm:2connect} it is 2-connected. Hence, by Theorem \ref{thm:commute} $\qt(X):=\overline R(\skel\Bc(X))$ is an NBQP. Let $\vu:V(\qt(X))\rightarrow X$ be the natural vector labeling. Consider now how the radial graph construction appears geometrically. We know that a given vertex $v$ is incident to a face $F_w$ if and only if $vw$ is a diameter pair. Hence $R(\skel\Bc(X))$, the radial graph of the skeleton, is the bipartite graph on vertex classes $X$ and $\{F_x|x\in X\}$ with edge set $\{(v,F_w) | v,w\in X \text{ and }|v-w|=1 \}$. By Theorem \ref{thm:canonicalselfduality} we know that in the canonical self-duality $v$ is sent to $F_v$. Therefore, the reduced radial graph $\qt(X)$ is indeed the diameter graph of $X$. By Lemma \ref{lem:precise} $(a)$, it is also neighborhood-coned.

    Finally, recall the rotation definition for an NBQP (Section \ref{sec:rotation}). We lift the quadrangulation to the sphere and consider the rotation in one of the color classes. We have also seen that the lifted quadrangulation map $\widetilde Q$ is the same as the radial map of the diagonal graph of the NBQP; this is exactly Theorem \ref{thm:commute} $(c)$.  Hence, in this case, a possible rotation system for $\qt(X)$ is given by the rotation of the vertices of $R(\skel\Bc(X))$ that correspond to the vertex class $X$. That is, for part $c)$ we have to show that faces of $\Bc(X)$ incident to a given vertex $v$ correspond to the vertices of $X$ that are diameter far from $v$ and the faces appear in the right order. This follows immediately from Lemma \ref{lem:precise} $(b)-(c)$.
    
\end{proof}

Therefore, we have an equivalent definition for $\qt(X)$. It is the unique NBQP map such that there is a vector labeling $\vu$ that satisfies the conditions of Lemma \ref{lem:diameterispolyhedral}.

\section{Generating NBQP maps and SISD maps}\label{sec:generating}
Here we discuss a generating method for NBQP maps and 2-connected SISD maps. Let $Q$ be a quadrangulation of a surface and let $f=abcd$ be a face of it. We say that the face $f$ is \emph{contractible} at the diagonal $ac$ if we can shrink it as shown in Figure \ref{fig:facecontr} until $a$ and $c$ coincide, and we get a simple quadrangulation. That is, $ac$ is contractible if we do not get loops or parallel edges, i.e. $ac$ is not an edge and there are exactly two common neighbors of $a$ and $c$. We will denote the resulting new vertex by $\overline{ac}$. If a face is contractible at a diagonal $ac$, we denote by $Q/ac$ the contracted quadrangulation. Although a pair of vertices might form a diagonal in more than one face, in those cases it is not contractible, hence $Q/ac$ is well-defined.  

\begin{figure}[!h]
    \centering
     \includegraphics[width=0.6\linewidth]{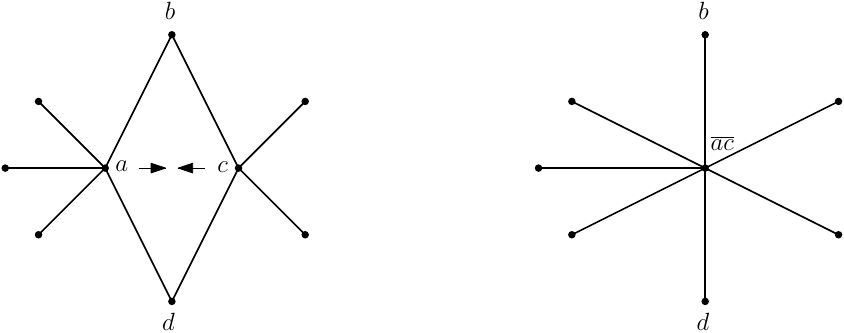}
    \caption{Face contraction operation.}
    \label{fig:facecontr}
\end{figure}

The inverse of a face contraction is called a \emph{vertex splitting}. Here, we can pick any two edges incident to a common vertex and split them into a new face.  Negami and Nakamoto showed the following theorem.

\begin{lemma}[Negami and Nakamoto \cite{NEG93}] \label{lem:contractible}
    Each simple quadrangulation of the projective plane has a contractible diagonal except $Q_4$ and $Q_{3,4}$ embedded as in Figure \ref{fig:q4q34}.
\end{lemma}

\begin{figure}[!h]
    \centering
    \includegraphics[width=0.2\linewidth]{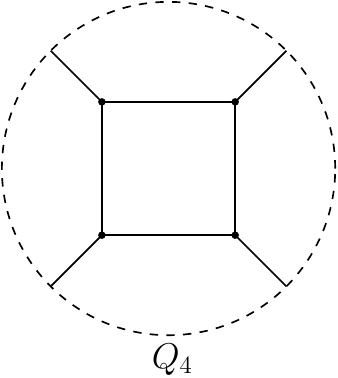}\hspace{20px}
    \includegraphics[width=0.2\linewidth]{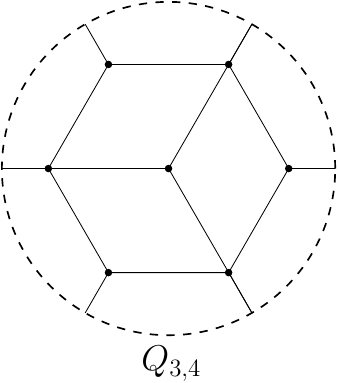}
\caption{The only quadrangulations of $\Pb^2$ that have no contractible faces.}
    \label{fig:q4q34}
\end{figure}

This gives us a generating method for quadrangulations. We can reach any desired quadrangulation with the right sequence of vertex splittings either from $Q_4$ or from $Q_{3,4}$. 
As a face contraction can only increase the chromatic number of a quadrangulation, we obtain the following. 

\begin{corollary} \label{cor:contractible2}
    If $Q$ is an NBQP map that is not $Q_4$, then $Q$ has a contractible diagonal $ac$ and $Q/ac$ is also an NBQP. In other words, any NBQP can be generated from $Q_4$ using vertex splitting.
\end{corollary}

 The colorings of projective quadrangulations were first studied by Youngs, who showed the following surprising result. 

\begin{theorem}[Youngs \cite{YOU96}]\label{thm:youngs}
    If $Q$ is a quadrangulation of the projective plane, then the chromatic number of $Q$ is either 2 or 4.
\end{theorem}

From Lemma \ref{lem:contractible} and Youngs' result it follows that face contraction/vertex splitting never changes the chromatic number of a projective quadrangulation, and all NBQP maps have chromatic number 4. (See Section \ref{sec:furtherconnections} for more results on colorings.)

Bracho et al. introduced a similar generating procedure for 3-connected SISD graphs \cite{bracho2020strongly}. It is essentially the same as vertex splitting for NBQPs combined with Theorem \ref{thm:commute}. The main difference is that they restrict themselves to the 3-connected world. 

If $G$ is a planar map and $e=vw$ is an edge of it, then let $G\setminus e$ denote the map obtained by deleting $e$. Let $G/e$ be the map obtained from $G$ by first deleting $e$ and then contracting $v$ and $w$ into a single vertex. When contracting $v$ and $w$ we do not merge edges going to common neighbors of $v$ and $w$, but keep them as parallel edges.  

They showed that in every 3-connected strongly involutive self-dual map $G$ other than the odd wheel graphs there exists a dual edge pair $e,\tau_G(e)$ such that $(G/ e)\setminus \tau_G(e)$ is a 3-connected strongly involutive self-dual map. We step outside of the 3-connected world, hence we formulate the generating method in the more general case.

Suppose $Q$ is an NBQP and we contract a quadrangle $abcd$ along $ac$. Then, in $\diag(Q)$ the edge $ac$ is contracted and the edge $bd$ is deleted. Therefore, combining Lemma \ref{cor:contractible2} and Theorem \ref{thm:commute} we obtain the following.

\begin{corollary}\label{cor:generatingsisdgraphs}
     In every 2-connected strongly involutive self-dual map $G$ other than $K_4$ there exists a dual edge pair $e,\tau_G(e)$ such that $(G/ e)\setminus \tau_G(e)$ is a 2-connected strongly involutive self-dual map. See Figure \ref{fig:planarsplit} for an example. 
\end{corollary}

\begin{figure}[!ht]
    \centering
    \includegraphics[width=0.80\linewidth]{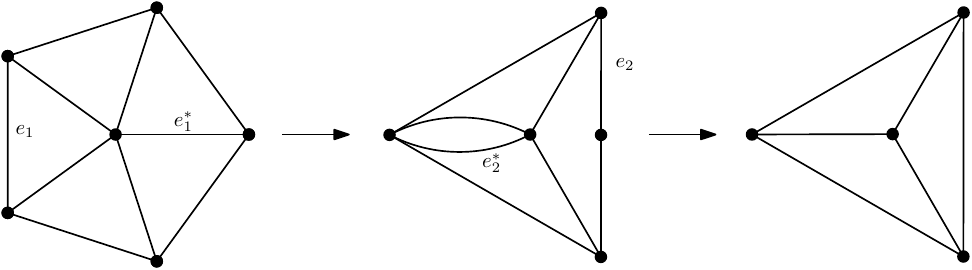}
    \caption{Reducing a strongly involutive self-dual graph to $K_4$.}
    \label{fig:planarsplit}
\end{figure}

\section{Rigidity}\label{sec:rigidity}
\subsection{Rigid bar frameworks}
The proof of Theorem \ref{thm:main_nbqp_to_diam} relies on ideas from rigidity theory. In this section we introduce the necessary notions and show that  neighborhood-coned frameworks of NBQP maps are stress-free. 

For all notions and basic results we follow Lovász \cite{lovasz2019graphs}; see Chapters 14 and 15 therein for a more detailed introduction. Recall that a vector labeling of a graph $G=(V,E)$  is a function $\vu:V\rightarrow \R^d$. We can think of it as specifying the position of the vertices of $G$ in $\R^d$. A pair $(G,\vu)$ is called a \emph{framework} if $\vu_i\ne\vu_j$ when $ij\in E$. One can think of a framework as a collection of rigid bars corresponding to the edges of $G$ that are joined at the endpoints. The bars are rigid, but we can change the position of the vertices, as long as the length of the bars remains unchanged.  

We note that often the vertices of a framework are assumed to be in some kind of general position, but this is not the case for us. We will need frameworks where two vertices are mapped to exactly the same position. On the other hand, we do not allow bars of length 0, hence the requirement that $\vu_i\ne\vu_j$ when $ij\in E$.

We say that $(G,\vu)$ is flexible if we can specify a \emph{proper flexing}, i.e., a smooth function $\vu_i:[0,1]\to\R^d$ for each $i\in V$ such that $\vu_i(0)=\vu_i$ and $|\vu_i(t)-\vu_j(t)|$ is constant for each $ij\in E$, but not constant for at least one pair of nonadjacent vertices $i,j\in V$. We say that $(G,\vu)$ is \emph{locally rigid} if no such proper flexing exists. A more general notion of movement is the notion of infinitesimal motions. A vector labeling $\vv: V\rightarrow \R^d$ is an \emph{infinitesimal motion} of $(G,\vu)$ if the following holds for every $ij\in E$.

\begin{equation}\label{eq:infinitesimal}
    (\vu_i-\vu_j)^\top (\vv_i -\vv_j)=0
\end{equation}

The initial velocities of a proper flexing form an infinitesimal motion, but not all infinitesimal motions correspond to a proper flexing. We will denote by $\mathbf{Inf}_{(G,\vu)}$ the linear space of infinitesimal motions of $(G,\vu)$. We say that $\vv$ is a \emph{rigid infinitesimal motion} if it corresponds to the initial velocities of rigid motion, that is, if $\vv_i=A\vu_i+\mathbf{b}$
for some skew-symmetric matrix $A\in\R^{d\times d}$ and some vector $\mathbf{b}\in\R^d$. A framework is \emph{infinitesimally rigid} if it has only rigid infinitesimal motions. 

A function $S:E\rightarrow \mathbb{R}$ is a \emph{stress} on $(G,\vu)$ if 

\begin{equation}\label{eq:stress}
\sum\limits_{j\in N(i)} S_{ij}(\vu_j-\vu_i)=0  
\end{equation}

for each $i\in V$. Note that we have $S_{ij} = S_{ji}$. We can think of a stress as forces acting on the edges of the framework. A force acting on an edge $ij$ is parallel to $\vu_j-\vu_i$, and therefore \eqref{eq:stress} means that each vertex is in equilibrium. (Note how this would break if we allowed $\vu_i=\vu_j$ for $ij\in E$). A framework $(G,\vu)$ is \emph{stress-free} if it carries only the identically-zero stress. We denote by $\mathbf{Str}_{(G,\vu)}$ the linear space of stresses of $(G,\vu)$.

The \emph{rigidity matrix} of $(G,\vu)$ is the matrix $R=R_{(G,\vu)}\in \R^{|E|\times d\cdot|V|}$ whose rows correspond to edges and the columns in groups of $d$ correspond to vertices. In the row of an edge $ij$ we have zeros except for the $d$ columns of $i$ and the $d$ columns of $j$, in those positions we have the coordinate values of $\vu_i-\vu_j$ and $\vu_j-\vu_i$ respectively. As $\vu_i\neq \vu_j$ when $ij\in E$, we do not have all zero rows. The right null space of $R$ corresponds to infinitesimal motions, and the left null space corresponds to stresses. This implies the following. 

\begin{equation}\label{eq:dimension}
 rk(R) = d|V|-dim(\mathbf{Inf}_{(G,\vu)}) = |E|-dim(\mathbf{Str}_{(G,\vu)})
\end{equation}

Intuitively, if we have a stress on a framework, it limits the possible movements of the framework. On the other hand, in a stress-free framework any non-rigid infinitesimal motion can be turned into a proper flexing. 

    \begin{lemma}[Lemma 15.7. from \cite{lovasz2019graphs}]\label{lem:flexing}
        If a stress-free framework is not infinitesimally rigid, then it is flexible. Furthermore, if we fix a non-rigid infinitesimal motion $\vv$, then there is a proper flexing where the initial velocities of the vertices are given by $\vv$.
    \end{lemma}

One of the classical results on stress-freeness is due to Cauchy. He showed that the framework given by the skeleton of a 3-dimensional polytope is stress-free. In a planar or projective-planar map whose edges are 2-colored, we call a node \emph{quiet} if the edges of each color emanating from it are consecutive in the rotation of the vertex in the embedding. The proof of Cauchy's result is based on the existence of a quiet node in any 2-colored planar map. We will need the following variant of this idea in the projective plane. 

\begin{lemma}[Exercise 2.6 in \cite{lovasz2019graphs}]\label{lemma:quiet}
    In every simple graph embedded in the projective plane whose edges are 2-colored, there is at least one quiet node.
\end{lemma}

\subsection{Rigidity of NBQP maps in $\R^3$}

The following analog of Cauchy's theorem plays an important part in the proof of Theorem \ref{thm:main_nbqp_to_diam}.  We show that if we have a neighborhood-coned framework in $\R^3$ which corresponds to a NBQP map, then it is stress-free. 

\begin{lemma}\label{lem:stressfree}
    Suppose $Q=(V,E)$ is a NBQP map with a rotation system $\rho$ and $\vu:V\rightarrow \R^3$ is a vector labeling of $Q$ with the following properties:
    \begin{enumerate}[label=\alph*)]
        \item $(Q,\vu)$ is neighborhood-coned.
        \item The cone rotation system of $(Q,\vu)$ matches the rotation system $\rho$.
    \end{enumerate}
    Then $(Q,\vu)$ is stress-free.
\end{lemma}
\begin{proof}
    We adopt the steps of the famous proof of Cauchy's stress-freeness theorem. Suppose $(Q,\vu)$ carries a nonzero stress $S$. Color the edges of $Q$ red where $S_{ij}>0$ and blue if $S_{ij}<0$ and also apply this coloring to the bars of the framework. If we delete all edges where $S_{ij}=0$, then we end up with a 2-colored graph. From Lemma \ref{lemma:quiet} we see that there must be a quiet vertex in this graph, that is, a vertex $i$ where the red (and blue) edges are consecutive around $i$. As  $(Q,\vu)$ is neighborhood-coned and the cone rotation system matches the rotation system $\rho$, this means that there is a plane $H$ that goes through $\vu_i$ and separates the red and blue edges emanating from $\vu_i$, see Figure \ref{fig:stressfreeproof}. The idea is to show that $i$ cannot be in equilibrium, as all edges are pushing or pulling it into the same half-space bounded by $H$. Let $\vn$ be the normal vector of $H$ pointing to the half-space containing the red edges. For each red edge $ij$ we have $\vn^{\top}(\vu_j-\vu_i)>0$ and hence $S_{ij}\vn^{\top}(\vu_j-\vu_i)>0$. On the other hand, for each blue edge $\vn^{\top}(\vu_j-\vu_i)<0$, which also implies $S_{ij}\vn^{\top}(\vu_j-\vu_i)>0$. For the uncolored edges incident to $i$, we simply have $S_{ij}\vn^{\top}(\vu_j-\vu_i)=0$. Therefore, 
    $\sum\limits_{j\in N(i)}S_{ij}\vn^{\top}(\vu_j-\vu_i)>0$  and hence $\vn^{\top}( \sum\limits_{j\in N(i)}S_{ij}(\vu_j-\vu_i))>0$.     As $S$ is a stress, we have $\sum\limits_{j\in N(i)}S_{ij}(\vu_j-\vu_i)=0$, a contradiction. 
  
\begin{figure}[!ht]
    \centering
    \includegraphics[width=0.3\linewidth]{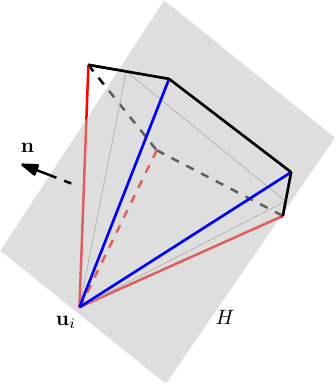}
    \caption{A quiet node cannot be in equilibrium. }
    \label{fig:stressfreeproof}
    
\end{figure}
    
\end{proof}

    Another way to show Lemma \ref{lem:stressfree} is by using the 3-degeneracy of the diameter graph, which follows from the fact that any subgraph $Q'=(V',E')$ of an NBQP graph  has at most $2|V'|-2$ edges (See Lemma 4.4 in Part I). If we have a vertex of degree at most 3 so that the vectors pointing to its neighbors are independent, then we must have 0 on these edges in any stress. Any three distinct extreme rays of a pointed cone in $\R^3$ have linearly independent directions. Hence, we can just delete such a vertex and the stress-space does not change. Repeatedly applying this shows that the only stress is the all zero stress. Although this approach is more elementary, we believe that the rigidity proof is more insightful.

\section{Existence of diameter graphs}\label{sec:qtexsists}
We are now ready to show Theorem \ref{thm:main_nbqp_to_diam}. That is, our goal is to show that for every NBQP map $Q=(V,E)$ there is an extremal set $X\subset \R^3$ such that $\qt(X)=Q$. By Lemma \ref{lem:diameterispolyhedral} we are looking for a vector labeling $\vu:V\rightarrow \R^3$ which is a diameter realization of $Q$ in the following sense:

\begin{enumerate}[label=\alph*)]
    \item $\vu$ is injective.
        \item $ij\in E(Q)$ if and only if $\vu_i\vu_j$ is a diameter of the point set $X_\vu=\{\vu_i\mid i\in [n] \}$.
        \item The cone rotation system of $(Q,\vu)$ matches one of the rotation systems of $Q$.
\end{enumerate}
    
We will prove the theorem inductively. As discussed in Section \ref{sec:generating}, we can generate any NBQP map starting from $Q_4$ and performing vertex splits. Hence it is enough to show that every vertex split can also be performed geometrically. Suppose we are splitting a vertex $v$ into $v_1,v_2$. We will first simply place both $v_1$ and $v_2$ at the original position of $v$. Then we will consider the rigid bar framework given by the desired diameter pairs. We show that there is a flexing of this framework that pulls $v_1$ and $v_2$ apart and results in the desired diameter realization. See Figure \ref{fig:steps} for an example.

\begin{figure}[!ht]
    \centering
    \includegraphics[width=0.2\linewidth]{figures/q4.pdf}
    \includegraphics[width=0.2\linewidth]{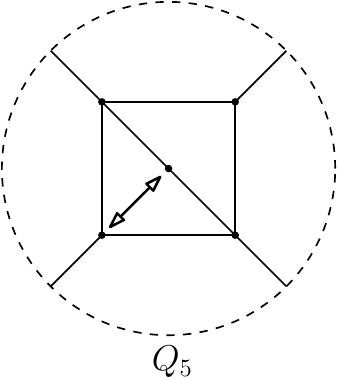}
    \includegraphics[width=0.2\linewidth]{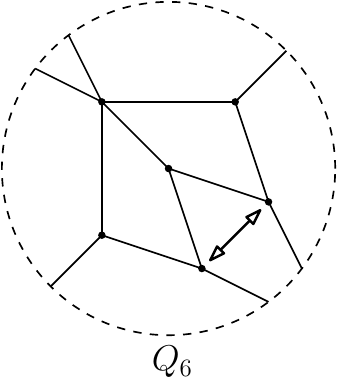}
    \hspace{100px}
    \includegraphics[width=0.2\linewidth]{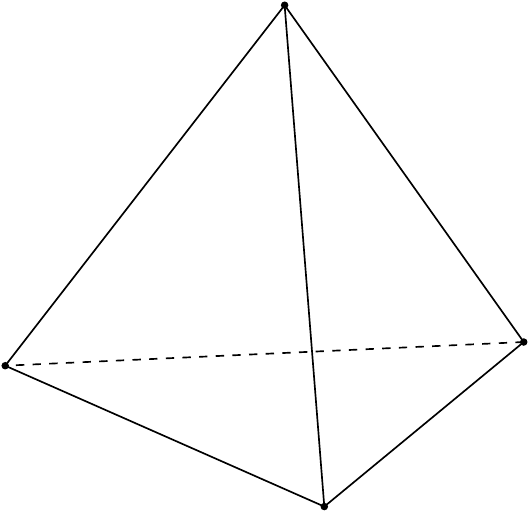}
    \includegraphics[width=0.2\linewidth]{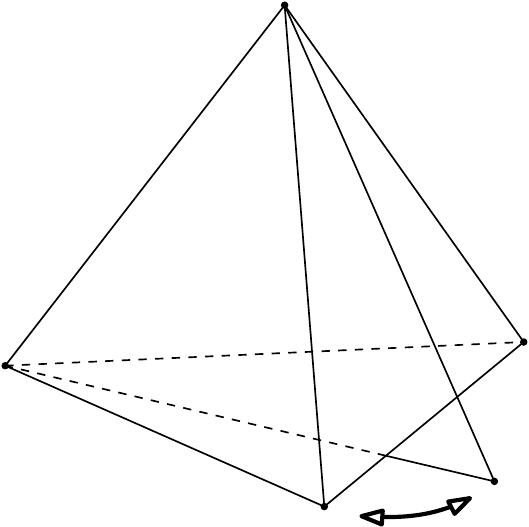}
    \includegraphics[width=0.2\linewidth]{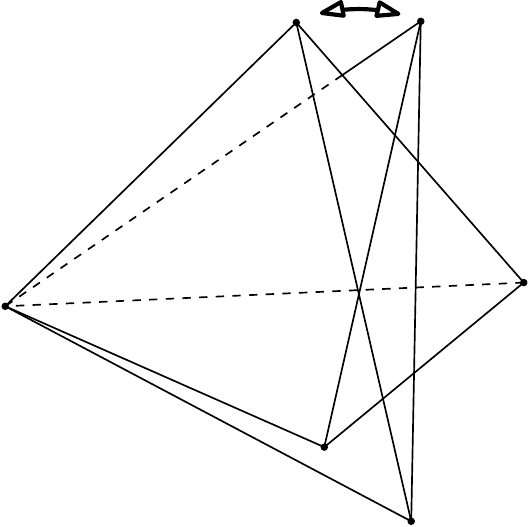}
    
    \caption{Vertex splitting on $\Pb^2$ and the corresponding steps in creating the diameter realization.}
    \label{fig:steps}
\end{figure}

\begin{proof}[Proof of Theorem \ref{thm:main_nbqp_to_diam}]

     We proceed by induction on the number of vertices. For $n=4$ there is only one quadrangulation of the projective plane, namely $Q_4$. We pick the rotation system $\rho$ shown in Figure \ref{fig:basecase}. If we map the vertices to the four vertices of a regular tetrahedron, as shown in Figure \ref{fig:basecase}, we get the desired diameter quadrangulation and we also get the right rotations around the neighborhood-cones.   

\begin{figure}[!ht]
    \centering
    \includegraphics[width=0.7\linewidth]{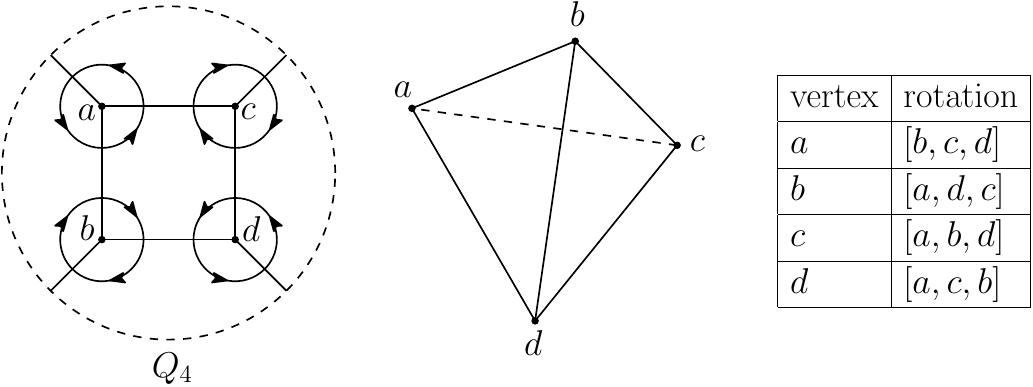}
    \caption{The base case, $Q_4$ and the corresponding extremal point set, the regular tetrahedron.}
    \label{fig:basecase}
\end{figure}

    Assume that $n>4$ and the rotation system $\rho$ is fixed. By Corollary \ref{cor:contractible2} there is a quadrangle $a_1b_1a_2b_2$ such that the diagonal $a_1a_2$ is contractible.  As we can change the roles of $b_1$ and $b_2$, and also the roles of $a_1$ and $a_2$, we might assume the following.
    \begin{enumerate}
        \item The rotation of $a_1$ in $\rho$ is  $[b_1,i_1,\dots i_k,b_2]$, and the rotation of $a_2$ is $[b_2,j_1,\dots, j_t,b_1]$ for some $i_1,\dots,i_k,j_1,\dots j_t\in V$.
        \item In the rotation of $b_1$ in $\rho$ the vertex $a_2$ is followed by $a_1$ and in the rotation of $b_2$ the vertex $a_1$ is followed by $a_2$.
    \end{enumerate}

     Here $k$ or $t$ may be zero. The second condition fixes the choice also when $k=t=0$. At a vertex of degree two, we use the unique cyclic order on its two neighbors.
    
     Let $a$ denote the contracted vertex in $Q/a_1a_2$. Then there is a rotation system $\overline{\rho}$ of $Q/a_1a_2$ such that the rotation of $a$  is $[b_1,i_1,\dots i_k,b_2,j_1,\dots,j_t]$. In $\overline{\rho}$ the rotations of $b_1$ and $b_2$ are obtained by replacing $a_1$ and $a_2$ with a single $a$. For the rest of the vertices, the  rotation in $\overline{\rho}$ is obtained by replacing $a_1$ and $a_2$ with $a$, if one of them is present. (Note that since $a_1a_2$ is a contractible diagonal, the only common neighbors of $a_1$ and $a_2$ are $b_1$ and $b_2$). Let $N_1=N(a_1)\setminus\{b_1,b_2\}$ and $N_2=N(a_2)\setminus\{b_1,b_2\}$. See Figure \ref{fig:rotchanges}. 
    
    \begin{figure}[!h]
        \centering
            \includegraphics[width=0.6\linewidth]{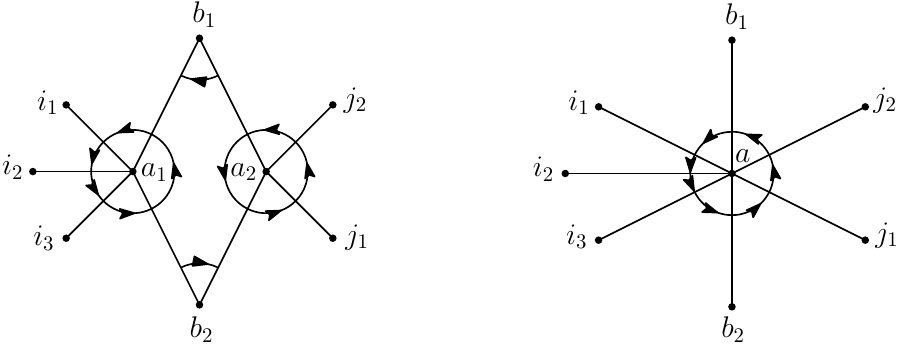}
        \caption{Changes in the rotation after contracting a diagonal. Note the direction at $b_1$ and $b_2$.}
        \label{fig:rotchanges}
    \end{figure}
    
    Note that if we have a realization for a given rotation system, then we also have a realization for the other rotation system. We can simply reflect the realization, and every cone rotation reverses. Hence, by induction there exists a vertex labeling $\vu'$ that is a diameter realization of $Q/a_1a_2$ satisfying $a),b)$ and $c)$, with cone rotation system $\overline{\rho}$ and diameter equal to $1$. Now consider the framework $(Q,\vu)$, where $\vu_i=\vu'_i$ for $i\in V\setminus\{a_1,a_2\}$, and $\vu_{a_1}=\vu_{a_2}=\vu'_{a}$.   That is, we split $\vu_{a}$ into two vertices and distribute the neighbors according to the rotation. As $a_1a_2$ is a contractible diagonal, $a_1a_2$ is not an edge and $\vu_{a_1}=\vu_{a_2}$ is not a problem in the sense that $(Q,\vu)$ is still a framework.

    Note that $\vu$ almost satisfies properties $a),b)$ and $c)$. For example each edge of $Q$ is represented by a pair of points at distance 1 and the diameter of the point set $X_\vu$ is 1. The problem is that $\vu_{a_1}$ and $\vu_{a_2}$ are in the same position, which is not allowed. This also means that some pairs of points are at distance 1 which should not be. These are the pairs of the form $(\vu_{a_1},\vu_j)$ where $j\in N_2$ and the pairs $(\vu_{a_2},\vu_i)$ where $i\in N_1$. To remedy this problem we only need to find a proper flexing of the framework $(Q,\vu)$, such that it moves $a_1$ and $a_2$ apart, the diameter does not increase, the positions of the other vertices stay distinct, and the cone rotation system of the resulting framework matches the rotation system $\rho$.  

    We break the argument for the existence of such a flexing into a sequence of lemmas. First we show that $(Q,\vu)$ is stress-free. Then we show that there is an infinitesimal motion that pulls $a_1$ and $a_2$ apart in the right relative direction. Then we turn the infinitesimal motion into a real flexing. Finally we show that a small initial part of the flexing satisfies the required properties $a),b)$ and $c)$.

Let $H$ denote the plane through $\vu'_a,\vu_{b_1},\vu_{b_2}$, and fix the normal vector of $H$ given by $\vn=(\vu_{a_1}-\vu_{b_1})\times(\vu_{a_1}-\vu_{b_2})$. These three points are non-collinear, hence $\vn\neq\mathbf 0$. Let
    \[
        H_1=\{\vy\in\R^3\mid\vn^\top(\vy-\vu_{a_1})>0\},
        \qquad
        H_2=\{\vy\in\R^3\mid\vn^\top(\vy-\vu_{a_1})<0\}.
    \]
    That is, $H_1$ and $H_2$ are the two half-spaces bounded by $H$. The cone rotation $[b_1,i_1,\dots,i_k,b_2,j_1,\dots,j_t]$, viewed counterclockwise from $\vu'_a=\vu_{a_1}$, gives $\vu_i\in H_1$ for $i\in N_1$ and $\vu_j\in H_2$ for $j\in N_2$, see Figure  \ref{fig:geosplit}.

    \begin{figure}[!ht]
        \centering
        \includegraphics[width=0.4\linewidth]{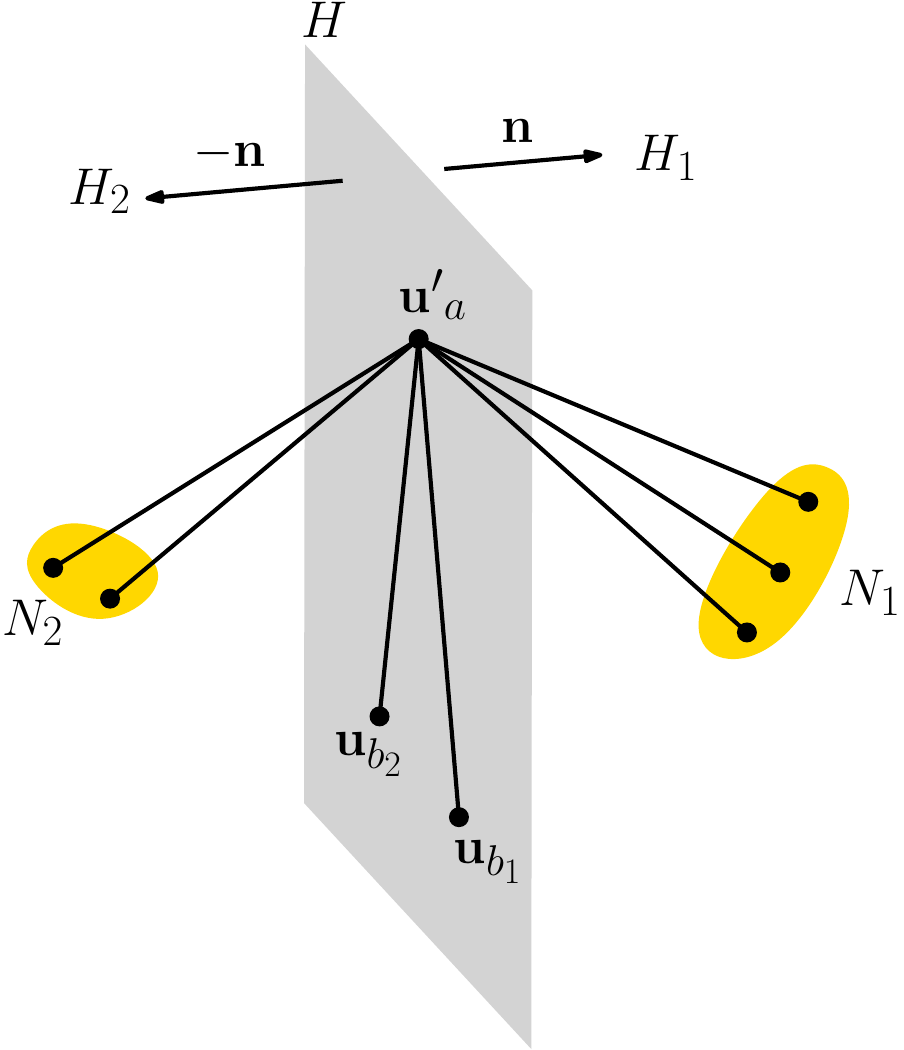}\hspace{20px}
        \includegraphics[width=0.4\linewidth]{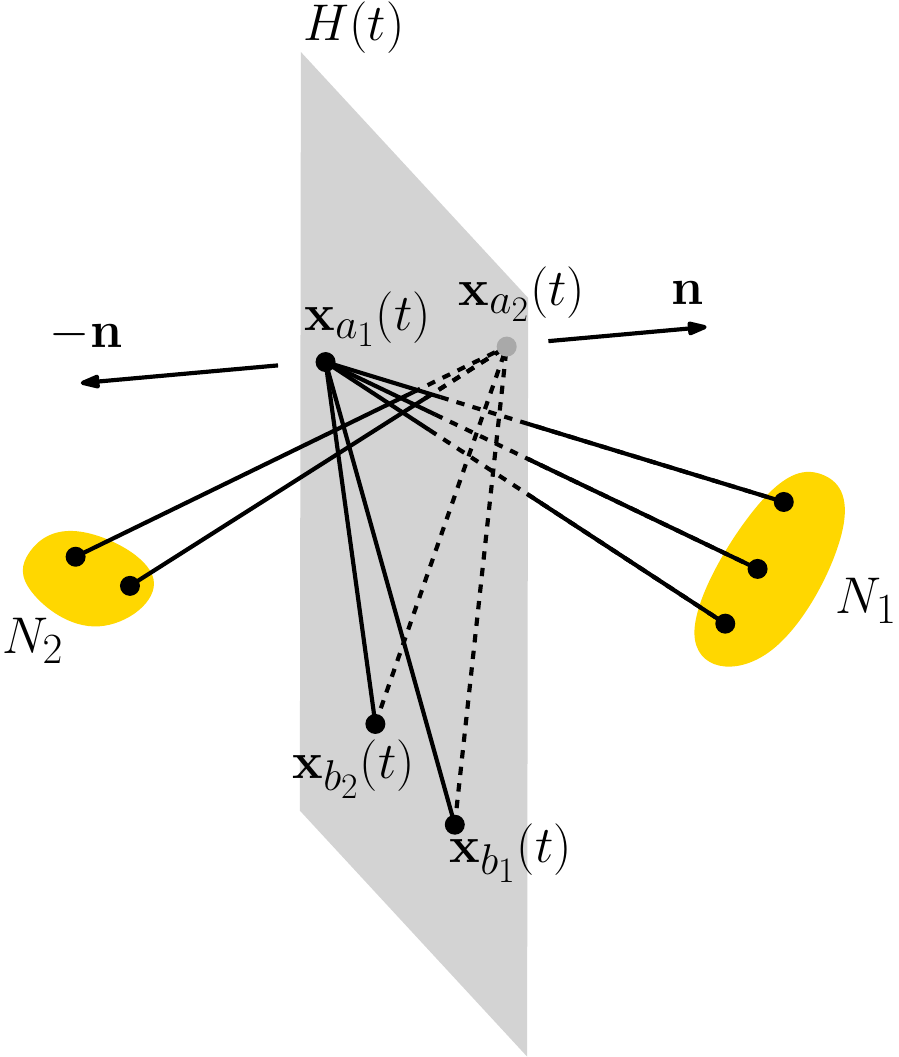}
        \caption{Pulling $a_1$ and $a_2$ apart. $a_1$ moves closer to $N_2$ and $a_2$ moves closer to $N_1$.}
        \label{fig:geosplit}
    \end{figure}

    \begin{lemma}\label{lem:splitstressfree}
        $(Q,\vu)$ is stress-free.
    \end{lemma}
    \begin{proof}
     Suppose that $S:E \rightarrow\mathbb{R}$ is a stress on $(Q,\vu)$. Let $S'$ be the modified version of $S$ where we set $S'_{ab_1}=S_{a_1b_1}+S_{a_2b_1}$ and $S'_{ab_2}=S_{a_1b_2}+S_{a_2b_2}$. Clearly if $S$ is a stress for $(Q,\vu)$, then $S'$ is a stress for $(Q/a_1a_2,\vu')$. On the other hand $(Q/a_1a_2,\vu')$ is stress-free by Lemma \ref{lem:stressfree}, hence $S'$ must be the zero-stress, i.e. $S_{a_1b_1}+S_{a_2b_1}=0$, $S_{a_1b_2}+S_{a_2b_2}=0$ and $S$ is zero on the other edges. By the definition of a stress, this implies that $S_{a_1b_1}(\vu_{b_1}-\vu_{a_1})+S_{a_1b_2}(\vu_{b_2}-\vu_{a_1})=0$. As $\vu_{b_1}-\vu_{a_1}$ and  $\vu_{b_2}-\vu_{a_1}$ are not parallel $S$ is also zero on $a_1b_1$ and $a_1b_2$ and similarly on $a_2b_1,a_2b_2$. Hence $S$ is the zero-stress. Therefore, $(Q,\vu)$ is stress-free. 
    
    \end{proof}
    We note that the previous lemma is essentially the same as Whiteley's argument in \cite{whiteley1990vertex} where he studies a similar vertex splitting procedure for generic frameworks. Our framework is very much non-generic, hence we cannot directly apply those results. 
     
     \begin{lemma}\label{lem:infinitesimalexists}
         There is a non-rigid infinitesimal motion $\vv$ of $(Q,\vu)$ such that $\vv_{a_1}-\vv_{a_2}\neq \mathbf{0}$. Furthermore, we can pick $\vv$ so that $\vn^\top (\vv_{a_1}-\vv_{a_2})<0$.
     \end{lemma}
    \begin{proof}
        If $\vv_{a_1}=\vv_{a_2}$ then $\vv$ is also an infinitesimal motion of $(Q/a_1a_2,\vu')$ by setting $\vv_a=\vv_{a_1}$. Hence it is enough to show that $(Q,\vu)$ has more non-rigid infinitesimal motions than $(Q/a_1a_2,\vu')$, that is we need that $\dim(\mathbf{Inf}_{(Q,\vu)})>\dim(\mathbf{Inf}_{(Q/a_1a_2,\vu')})$. This is intuitively clear: splitting the vertex should introduce 3 degrees of freedom, doubling the two edges takes away only 2 of those. As the frameworks are stress-free we can calculate these dimensions. Since $(Q,\vu)$ and $(Q/a_1a_2,\vu')$ are stress-free, we have $\dim(\mathbf{Str}_{(Q,\vu)})=0$ and $\dim(\mathbf{Str}_{(Q/a_1a_2,\vu')})=0$. Plugging this into \eqref{eq:dimension} we obtain

    \begin{equation*}
        \dim(\mathbf{Inf}_{(Q,\vu)})=3n-(2n-2)+\dim(\mathbf{Str}_{(Q,\vu)})=n+2,
    \end{equation*}
    \begin{equation*}
        \dim(\mathbf{Inf}_{(Q/a_1a_2,\vu')})=3(n-1)-(2n-4)+\dim(\mathbf{Str}_{(Q/a_1a_2,\vu')})=n+1.
    \end{equation*}

     Hence, we have found an infinitesimal motion $\vv$ of $(Q,\vu)$ such that $\vv_{a_1}-\vv_{a_2}\neq \mathbf{0}$. We need to turn this into a proper flexing, but there is a last adjustment to make. It is important that we pull $a_1$ and $a_2$ apart in the right direction.
    
    Recall that $\vn=(\vu_{a_1}-\vu_{b_1})\times(\vu_{a_1}-\vu_{b_2})$, that is, it points to the halfspace that contains $N_1$. Since $\vv$ is an infinitesimal motion we have $(\vv_{b_1}-\vv_{a_1})^\top(\vu_{b_1}-\vu_{a_1})=0$ and $(\vv_{b_1}-\vv_{a_2})^\top(\vu_{b_1}-\vu_{a_2})=0$. Using that $\vu_{a_1}=\vu_{a_2}$ the second can be rewritten as $(\vv_{b_1}-\vv_{a_2})^\top(\vu_{b_1}-\vu_{a_1})=0$. Hence 

        $$(\vv_{b_1}-\vv_{a_1})^\top(\vu_{b_1}-\vu_{a_1})-(\vv_{b_1}-\vv_{a_2})^\top(\vu_{b_1}-\vu_{a_1})=0$$

        This simplifies to $(\vv_{a_2}-\vv_{a_1})^\top(\vu_{b_1}-\vu_{a_1})=0$. Similarly we get $(\vv_{a_2}-\vv_{a_1})^\top(\vu_{b_2}-\vu_{a_1})=0$, that is, $(\vv_{a_2}-\vv_{a_1})$ is orthogonal to both  $\vu_{b_1}-\vu_{a_1}$ and $\vu_{b_2}-\vu_{a_1}$. Therefore $(\vv_{a_2}-\vv_{a_1})$ is parallel to $\vn$. As $\vn$ and $(\vv_{a_1}-\vv_{a_2})$ are nonzero and parallel, we have $\vn^\top (\vv_{a_1}-\vv_{a_2})\ne 0$. If $\vn^\top (\vv_{a_1}-\vv_{a_2})<0$ we keep $\vv$, if $\vn^\top (\vv_{a_1}-\vv_{a_2})>0$ we will work with $-\vv$ instead, which is also an infinitesimal motion of $(Q,\vu)$. In other words, we can assume that  $\vn^\top (\vv_{a_1}-\vv_{a_2})<0$,  equivalently $\vv_{a_1}-\vv_{a_2}=-\lambda\vn$ for some $\lambda>0$.
      \end{proof}

    \begin{lemma}\label{lem:injective}
        There is a proper flexing $\vx$ of $(Q,\vu)$ with initial velocity satisfying Lemma \ref{lem:infinitesimalexists}. Furthermore, there is an $\varepsilon_1>0$ such that for any $0<t<\varepsilon_1$ the vector labeling $\vx(t)$ is injective and the points $\vx_{b_1}(t),\vx_{b_2}(t),(\vx_{a_1}(t)+\vx_{a_2}(t))/2$ are not collinear.
    \end{lemma}
    \begin{proof}
    Pick a non-rigid infinitesimal motion $\vv$ according to Lemma \ref{lem:infinitesimalexists}. To turn $\vv$ into a proper flexing we apply Lemma \ref{lem:flexing}. As we have seen $(Q,\vu)$ is stress-free, hence  there is a proper flexing $\vx(t)$ on  $[0,1]$ such that $\vx(0)=\vu$ and the first derivative of $\vx$ at $0$ is $\vv$. As the three points $\vx_{b_1}(t),\vx_{b_2}(t),(\vx_{a_1}(t)+\vx_{a_2}(t))/2$ are not collinear at $t=0$, they stay non-collinear for small $t$. Furthermore, as $\vv_{a_1}-\vv_{a_2}\neq \mathbf{0}$, we can restrict the flexing to some interval $[0,\varepsilon_1]$ such that the vector labeling $\vx(t)$ is injective and the points $\vx_{b_1}(t),\vx_{b_2}(t),(\vx_{a_1}(t)+\vx_{a_2}(t))/2$ are not collinear, for any $0<t<\varepsilon_1$. 
        
    \end{proof}
    From now on we fix a flexing $\vx$ satisfying the previous lemma. Together with the previous lemma the next two lemmas show that for small $\varepsilon$ the vector labeling $\vx(\varepsilon)$ satisfies properties $a),b)$ and $c)$. For $t<\varepsilon_1$ let $H(t)$ denote the plane spanned by $\vx_{b_1}(t),\vx_{b_2}(t),(\vx_{a_1}(t)+\vx_{a_2}(t))/2$. As the three defining points are not collinear for any $t<\varepsilon_1$, this is well defined. Clearly $H(0)=H$ and for $t>0$ the plane $H(t)$ is the perpendicular bisector of  $\vx_{a_1}(t)$ and $\vx_{a_2}(t)$. 

    \begin{lemma}\label{lem:diametergraphcorrect}
        There is an $\varepsilon_2>0$ such that for any $0<\varepsilon<\varepsilon_2$ the diameter of $\vx(\varepsilon)$ is exactly 1 and the underlying diameter graph is the graph of $Q$.
    \end{lemma}
     \begin{proof}
         
      As $\vx(\varepsilon)$ is a flexing of $(Q,\vu)$, each edge of $Q$ is represented by a pair of points at distance 1. The diameter of $\vu'$, and therefore of $\vu$, is 1 by induction. Hence, if we choose $\varepsilon$ small, then for every pair of vertices $i,j\in V$ where $|\vx_i(0)-\vx_j(0)|<1$ we will still have $|\vx_i(\varepsilon)-\vx_j(\varepsilon)|<1$.

Therefore, the only remaining pairs of vertices where the diameter could increase are the $(a_1,j)$ pairs where $j\in N_2$ and the $(a_2,i)$ pairs where $i\in N_1$. These pairs have distance $1$ at $\vx(0)$ but they are not edges of $Q$, so their distance is not preserved by the framework. We must show that their distance decreases. At $t=0$, the points indexed by $N_1$ lie in $H_1$ and those indexed by $N_2$ lie in $H_2$, and these strict side relations remain true for any sufficiently small $\varepsilon>0$. Moreover,
      \[
          \frac{\vx_{a_1}(\varepsilon)-\vx_{a_2}(\varepsilon)}{\varepsilon}
          \longrightarrow \vv_{a_1}-\vv_{a_2}=-\lambda\vn,
          \qquad \lambda>0.
      \]
      Since $H(\varepsilon)$ is the perpendicular bisector plane, for sufficiently small $\varepsilon>0$ the point $\vx_{a_1}(\varepsilon)$ and the points of $\{\vx_{j}(\varepsilon)\mid j\in N_2\}$ lie strictly on one side of $H(\varepsilon)$ and the points $\vx_{a_2}(\varepsilon)$ and $\{\vx_{i}(\varepsilon)\mid i \in N_1\}$ lie strictly on the other side of $H(\varepsilon)$. Therefore, by the definition of $H(\varepsilon)$, the points of $N_2$ are strictly closer to $\vx_{a_1}(\varepsilon)$ than to $\vx_{a_2}(\varepsilon)$, see Figure \ref{fig:geosplit}. Their distance from $\vx_{a_2}(\varepsilon)$ is exactly 1, hence their distance from $\vx_{a_1}(\varepsilon)$ is strictly less than $1$. 

Similarly, the points of $N_1$ are closer to $\vx_{a_2}(\varepsilon)$ than to $\vx_{a_1}(\varepsilon)$. (If $N_1$ or $N_2$ is empty, there is nothing to check for that side.) Hence, the diameter of $\vx(\varepsilon)$ is exactly 1 and the diameter graph is $\gr (Q)$. As this holds for any small enough $\varepsilon>0$, we can fix a suitable $\varepsilon_2$. 
    \end{proof}

    \begin{lemma}\label{lem:checkingrotation}
        There is an $\varepsilon_3>0$ such that for any $0<\varepsilon<\varepsilon_3$ the vector labeling $\vx(\varepsilon)$ is neighborhood-coned and the cone rotations match the rotation system $\rho$. 
    \end{lemma}
    \begin{proof}

Throughout this proof take $0<\varepsilon<\min(\varepsilon_1,\varepsilon_2)$. 
If $\varepsilon>0$ is small enough, the rotation cannot change around any vertex except for $a_1,a_2,b_1$ or $b_2$ where we introduced/removed edges.
 Let $\Kc_v$ denote the neighborhood cone of vertex $v$. From the induction hypothesis, at $t=0$ the cones $\Kc_{a_1}$ and $\Kc_{a_2}$ have the correct rotation. Hence, if $\varepsilon$ is small, the rotation of the cones $\Kc_{a_1}$ and $\Kc_{a_2}$ will be correct. We only need to check the rotation of $\Kc_{b_1}$ and $\Kc_{b_2}$ and by symmetry it is enough to argue for $\Kc_{b_1}$. 

For sufficiently small $\varepsilon$, the old extreme rays of $\Kc_{b_1}$ other than the ray towards $\vu'_a$ remain extreme and keep their relative order. Hence we only need to show that $\vx_{a_1}(\varepsilon)-\vx_{b_1}(\varepsilon)$ and $\vx_{a_2}(\varepsilon)-\vx_{b_1}(\varepsilon)$ generate extreme rays of $\Kc_{b_1}(\varepsilon)$ and that, counterclockwise from the apex, the ray towards $\vx_{a_2}(\varepsilon)$ is followed directly by the ray towards $\vx_{a_1}(\varepsilon)$. These two rays converge to the old ray towards $\vu'_a$, so no other ray can occur between them locally.

 Let $L(\varepsilon)$ denote the perpendicular bisector plane of $\vx_{b_1}(\varepsilon)$ and $\vx_{b_2}(\varepsilon)$. Let $D(\varepsilon)$ be the disk $B(\vx_{b_1}(\varepsilon))\cap B(\vx_{b_2}(\varepsilon))\cap L(\varepsilon)$ and let $C(\varepsilon)$ be the bounding circle of $D(\varepsilon)$. See Figure \ref{fig:rotproof}. Note that extremal rays of $\Kc_{b_1}(\varepsilon)$ all intersect $L(\varepsilon)$. Hence, to find the extremal vectors of $\Kc_{b_1}(\varepsilon)$ it is enough to find the extremal points of the polygon $R(\varepsilon)=L(\varepsilon)\cap \Kc_{b_1}(\varepsilon)$. 

For every neighbor $j$ of $b_1$, the point $\vx_j(\varepsilon)$ lies on the unit sphere about $\vx_{b_1}(\varepsilon)$ and in the unit ball centered at $\vx_{b_2}(\varepsilon)$.

    Since $L(\varepsilon)$ is the perpendicular bisector plane of the two centers, it meets each extreme ray of $\Kc_{b_1}(\varepsilon)$ on the segment from $\vx_{b_1}(\varepsilon)$ to the corresponding neighbor. Each such intersection point lies in $B(\vx_{b_1}(\varepsilon))\cap L(\varepsilon)=D(\varepsilon)$. The section $R(\varepsilon)$ is the convex hull of these intersection points, so $R(\varepsilon)\subset D(\varepsilon)$. As $\vx_{a_1}(\varepsilon)$ and $\vx_{a_2}(\varepsilon)$ lie in $R(\varepsilon)\cap C(\varepsilon)$, they are extreme points of $D(\varepsilon)$ and hence also of $R(\varepsilon)$. Therefore they correspond to extreme rays of $\Kc_{b_1}(\varepsilon)$.

To see that they have the right order, let $\vc=(\vu_{b_1}+\vu_{b_2})/2$ be the center of $C(0)$. The definition of $\vn$ gives $\vn=(\vu_{b_2}-\vu_{b_1})\times(\vu'_a-\vc)$. Thus, viewed from $\vu_{b_1}$, the counterclockwise unit tangent to $C(0)$ at $\vu'_a$ is $-\vn/|\vn|$, see Figure \ref{fig:rotproof}. Since $\vv_{a_1}-\vv_{a_2}=-\lambda\vn$ with $\lambda>0$, and the tangent directions vary continuously, for sufficiently small $\varepsilon>0$ the point $\vx_{a_1}(\varepsilon)$ follows $\vx_{a_2}(\varepsilon)$ in counterclockwise order. At $b_2$ the viewing direction is reversed, so the counterclockwise tangent is $\vn/|\vn|$ and the order is $a_1$ followed by $a_2$, as required. This argument also applies when one or both of $N_1,N_2$ are empty.

    \begin{figure}[!ht]
        \centering
        \includegraphics[width=0.2\linewidth]{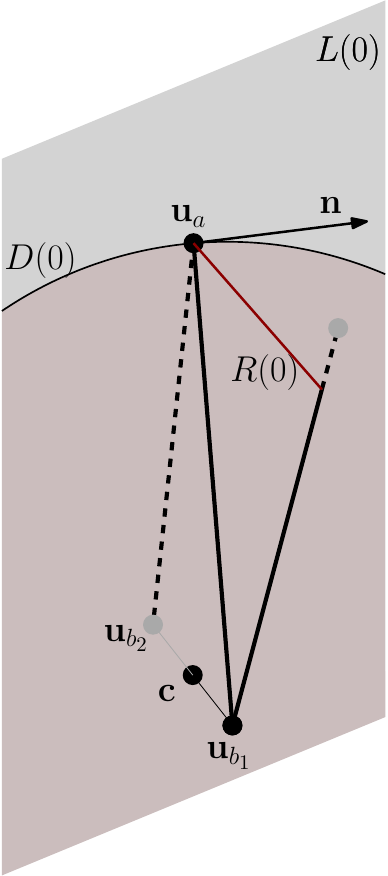}\hspace{20px}
        \includegraphics[width=0.2\linewidth]{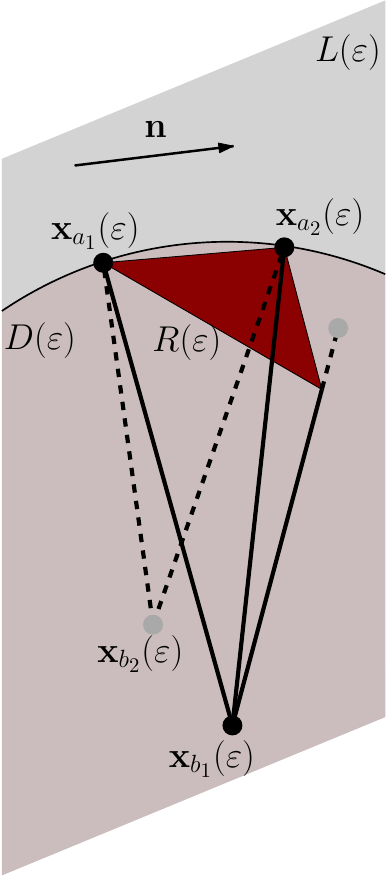}
        \caption{Changes in the rotation of the cone of $b_1$.}
        \label{fig:rotproof}
    \end{figure}
 \end{proof}

Now fix $0<\varepsilon< \min(\varepsilon_1,\varepsilon_2,\varepsilon_3)$ and consider the vector labeling $\vx(\varepsilon)$. Lemma \ref{lem:injective}, Lemma \ref{lem:diametergraphcorrect} and Lemma \ref{lem:checkingrotation} show that $\vx(\varepsilon)$ satisfies properties $a),b),c)$ respectively. This finishes the induction step; we have found a vector labeling for $Q$.

\end{proof}

Recall that Conjecture \ref{conj:montereuleaux} said that any simple 3-connected strongly involutive self-dual graph has a metric embedding. Theorem \ref{thm:main_nbqp_to_diam} settles this in a stronger form. Given a 2-connected (not necessarily simple) strongly involutive self-dual map $G$, the diameter realization of $\overline{R}(G)$ is a metric embedding of $G$. 

We are also ready to prove Theorem \ref{thm:main_skeleton}, which settles Conjecture \ref{conj:kupitz}. That is, we want to show that for every planar (not necessarily simple) 2-connected strongly involutive self-dual map $G$ there is an $X\subset \R^3$ such that $X$ is extremal for the Vázsonyi problem and the skeleton of $\Bc(X)$ is  isomorphic to $G$ as a map. 

Here the correspondence developed in Part I shows its usefulness. Instead of working with ball polytopes and SISD maps we can use Theorem \ref{thm:commute} to switch to diameter graphs and NBQP maps. For those Theorem \ref{thm:main_nbqp_to_diam} provides the geometric realization. Then we can use Theorem \ref{thm:commute} again to come back to SISD maps and ball polytopes again. 

\begin{proof}[Proof of Theorem \ref{thm:main_skeleton}]

Let $G$ be a 2-connected SISD map. By Theorem \ref{thm:commute} the quadrangulation $Q=\overline R(G)$ is an NBQP map. By Theorem \ref{thm:main_nbqp_to_diam} there is an extremal set $X\subset\R^3$ such that $\qt(X)=Q$. By Corollary \ref{cor:skeldiagonbody} we have $\skel\Bc(X)=\diag(\qt(X))=\diag(Q)=\diag(\overline R(G))$. Applying Theorem~\ref{thm:commute} again we obtain that $\diag(\overline R(G))=G$, hence the skeleton of $\Bc(X)$ is $G$, as desired. See Figure \ref{fig:mainsteps} for an illustration.
 
\end{proof}

We note that it is easy to miss some subtle parts of the previous argument. For example, it could be that there is a quadrangulation $Q'$ that has the same underlying graph as $Q$ but a different rotation system (see Figure 5 in Part I for an example). Simply taking a metric embedding of $G$ might result in an extremal set $X'$ such that $\qt(X')=Q'$ and not $Q$. This would give us a ball polytope with skeleton $\diag(Q')$ which is not $G$. This is why it is important not only that we create a point set whose diameter graph is the graph of $Q$, but also that we ensure that we get the correct rotations around the cones.

\section{Diameter graphs as generalized thrackles}\label{sec:thrackles}

In this Section we will consider all diameter graphs in $\R^3$, not just maximal ones.  A simple graph embedded in the projective plane is a \emph{parity embedding} when a simple cycle is essential if and only if it is of odd length. Recall from Part I the following folklore result.

\begin{lemma}[Lemma 6.4 in Part I]\label{lem:eqmaxcases}
    A simple graph $Q$ has a parity embedding in the projective plane if and only if it is a subgraph of an NBQP graph.

\end{lemma}

Parity embeddings also appear in the context of thrackles. 

\begin{theorem}[Cairns and Nikolayevsky  \cite{CN00}]\label{thm:thrpairtyemb}
    A graph is generalized thrackleable if and only if it has a parity embedding on the projective plane.
\end{theorem}

At this point we can claim the following characterization of maximal diameter graphs in $\R^3$.

\begin{theorem}
A graph on $n\ge4$ vertices is a diameter graph in $\R^3$ with $2n-2$ edges if and only if it is the graph of a simple non-bipartite quadrangulation of the projective plane.    
\end{theorem}
\begin{proof}
    Suppose $X$ is extremal and $|X|\ge 4$. By Theorem \ref{thm:2connect} and Corollary \ref{cor:canonicalissisd} we know that $\skel(\Bc(X))$ is 2-connected and strongly involutive self-dual. Let $G=\skel(\Bc(X))$. Hence, the diameter graph $\Gamma(X)=\gr(\qt(X))=\gr(\overline{R}(G))$ is an NBQP graph. The other direction is given by Theorem \ref{thm:main_nbqp_to_diam}.
 
\end{proof}
 
We show that this can be extended to the non-edge-maximal case. 

\begin{theorem}
    A graph on at least 2 vertices is a diameter graph in $\R^3$ if and only if it is a subgraph of an NBQP graph and it has at least one edge. 
\end{theorem}

\begin{proof}
    Let $Q$ be an NBQP map, and let $Q'$ be a subgraph of it with at least one edge. By Theorem \ref{thm:main_nbqp_to_diam} there is an extremal set $X$ so that $\qt(X)=Q$. By Lemma \ref{lem:stressfree} this diameter realization of $Q$ is stress-free. Since the framework is stress-free, every sufficiently small perturbation of its edge lengths admits a realization arbitrarily close to the original framework.  This follows from applying the Implicit Function Theorem to the squared-edge-length map, see 10.5 and 15.2.2 in \cite{lovasz2019graphs}. (Essentially, the stress-freeness condition expresses that the hypersurfaces corresponding to the edge length constraints meet transversally at the solution $X$.) 

    Therefore we can prescribe edge lengths 1 on $E(Q')$ and edge lengths strictly smaller than 1 on $E(Q)\setminus E(Q')$. This has a realization $X'\subset \R^3$ close to $X$.  If the perturbation is close enough to $X$, then the pairs of points that were strictly closer to each other than 1 are still strictly closer than 1. As $Q'$ has at least one edge, the diameter cannot decrease. Furthermore we can simply delete all vertices in $V(Q)\setminus V(Q')$. Hence, the diameter graph of $X'$ is exactly $Q'$. 

    In the other direction, we could argue the following way. Suppose $X$ is a set of points in $\R^3$. We could extend the ideas in Lemma \ref{lem:diameterispolyhedral} to show that the diameter graph $\Gamma(X)$ has a drawing in the projective plane which extends to an NBQP. Again the vertex rotations would correspond to the neighborhood rotations. This would require us to carefully analyze the skeleton of $\Bc(X)$. We choose an easier path: we can simply chain together the previous theorems. Any diameter graph in $\R^3$ is a generalized thrackleable graph (Theorem \ref{thm:diamisgent}). Any generalized thrackleable graph has a parity embedding in the projective plane (Theorem \ref{thm:thrpairtyemb}). Any parity embedding in the projective plane is a subgraph of an NBQP graph (Lemma \ref{lem:eqmaxcases}).  Hence, any diameter graph is a subgraph of an NBQP graph, and it must obviously have at least one edge since we have at least 2 vertices. 
       
\end{proof}

\section{Further connections and applications}\label{sec:furtherconnections}

Our paper relies heavily on the seminal paper of Kupitz, Martini and Perles \cite{kupitz2010ball}. They asked to characterize  simple 3-connected graphs that are skeletons of ball polytopes of extremal sets (Question 2 in \cite{kupitz2010ball}). Theorem \ref{thm:main_skeleton} answers this question, a simple 3-connected graph is equivalent to $\skel(\Bc(X))$ for some extremal $X$ if and only if it is strongly involutive self-dual.  

They also asked to characterize the extremal sets where the skeleton is simple and 3-connected. We say that $X\subset \R^3$ is a \emph{critical configuration} for the Vázsonyi problem if it is an extremal configuration and any point of $X$ is adjacent to at least 3 diameters. We also say that $X$ is \emph{strongly critical} if $X$ does not have a proper subset which is an extremal configuration. Example \ref{fig:example-G} is critical but not strongly critical.

\begin{question}[Question 3. in \cite{kupitz2010ball}]
     For which extremal sets $X\subset \R^3$ is $\Bc(X)$ 3-polytopal, i.e., has a face structure combinatorially equivalent to that of a 3-polytope?
\end{question}

They conjectured that this happens if and only if $X$ is strongly critical. This conjecture was recently confirmed by López-Campos,  Oliveros, and  Ramírez Alfonsín in \cite{lopez2023borsuk}. 
\begin{theorem}[Theorem 7 in \cite{lopez2023borsuk}]\label{thm:stronglycritical}
    If $X$ is extremal with $|X|\ge 4$, then $\skel\Bc(X)$ is simple and 3-connected if and only if $X$ is strongly critical. 
\end{theorem}
We obtain a new proof of this result. Let us recall the following two results from Part I. 

\begin{lemma}\label{lem:3conequiv}
    Let $G$ be a 2-connected SISD map. Then the following are equivalent:
    \begin{enumerate}
        \item $G$ is 3-connected.
        \item $\overline R(G)$ has no separating 4-cycles.
        \item For every proper subset $S\subset V(\overline R(G))$ that contains at least 4 vertices, we have $|E(\overline R(G)[S])|\le 2|S|-3$.
    \end{enumerate}
    
\end{lemma}

\begin{lemma}\label{lem:selfdualsimple}
    Any 3-connected self-dual graph is simple. 
\end{lemma}

With these results we have the following short proof of Theorem \ref{thm:stronglycritical}.

\begin{proof}[New proof of Theorem \ref{thm:stronglycritical}]
Note that  $X$ is strongly critical if and only if for any proper subset $S\subset X$ with at least 4 vertices we have at most $2|S|-3$ diameters spanned by $S$.

Lemma \ref{lem:3conequiv} and Lemma \ref{lem:selfdualsimple} tell us that the skeleton $G$ is simple and 3-connected, if and only if $\overline R(G)$ has no proper subset $S$ with at least 4 vertices and spanning more than  $2|S|-3$ edges. As $\overline R(G)=\qt(X)$ by Lemma \ref{lem:diameterispolyhedral}, this means that the skeleton is simple and 3-connected if and only if $X$ is strongly critical.
    
\end{proof}

\subsection{Colorings}
As we mentioned before Borsuk's conjecture can be viewed as a coloring question of diameter graphs. Hence, coloring properties of diameter graphs have been studied, see for example \cite{hujter2014multiple}. In \cite{lopez2023borsuk} the authors showed that strongly critical configurations correspond to the 4-critical diameter graphs. 

Interestingly, the coloring properties of NBQP maps are also very well studied. This started with the aforementioned theorem of Youngs (Theorem \ref{thm:youngs}). Later this was strengthened in various ways. As we have connected diameter graphs and NBQP maps, we can simply push results from one world to the other. For now, we decided to mention two such corollaries.  

Firstly, let $Q$ be a quadrangulation of the projective plane and let $c:V(Q)\rightarrow\{1,2,3,4\}$ be a proper vertex-coloring. A face of $Q$ is called \emph{rainbow} if its four vertices receive the four distinct colors. We say that a rainbow face is of \emph{type $t$}, where
$t\in\{2,3,4\}$, if the two vertices colored 1 and $t$ are opposite on the boundary of the face. Equivalently, the two remaining colors are also opposite.

\begin{theorem}[\cite{mohar2006quadrangulations} and also \cite{KaiserLoNakamotoNozakiOzeki2025}]\label{thm:bettercol}
Let $Q$ be a non-bipartite quadrangulation of the projective plane. Then, in every proper 4-coloring $c:V(Q)\to\{1,2,3,4\}$, there are at least three rainbow faces. More precisely, there is at least one rainbow face of each of the three types $2,3,4$.
\end{theorem}

This is also sharp in the sense that there is always a 4-coloring where we only have three rainbow faces. We can simply use the generating method from Section \ref{sec:generating} and maintain such a coloring. Indeed if we simply keep the original colors at each split, the coloring remains proper and the new face will never be rainbow.  

For extremal sets we have the following corollary of Theorem \ref{thm:bettercol} and Lemma \ref{lem:diameterispolyhedral}. 

\begin{corollary}
    If $X\subset \R^3$ is extremal such that $|X|\ge 4$, then in any proper 4-coloring $X\to\{1,2,3,4\}$ of $\qt(X)$, there are at least three rainbow faces of the diameter quadrangulation $\qt(X)$. More precisely, there is at least one rainbow face of each of the three types $2,3,4$.
\end{corollary}

Secondly, the local chromatic number of NBQP maps was also considered. A proper vertex-coloring of a graph $G$ is a local $k$-coloring if at most $k-1$ colors appear in the neighborhood of any vertex of $G$. The local chromatic number of $G$ is the smallest $k$ such that $G$ has a local $k$-coloring. Tardos, Simonyi and Mohar extended Youngs' theorem to local chromatic number.

\begin{theorem}[Theorem 1.3 in \cite{mohar2013local}]
    The local chromatic number of a quadrangulation of the projective plane is two or four.
\end{theorem}

A non-bipartite graph cannot have local chromatic number 2, hence NBQP maps have local chromatic number 4. 

\begin{corollary}
     If $X\subset \R^3$ is extremal such that $|X|\ge 4$, then in any proper $k$-coloring $X\to\{1,\dots, k\}$ of the diameter graph $\Gamma(X)$, there is a vertex $v$ and three neighbors of $v$ such that they all have different colors. 
\end{corollary}
This means that not only do we need 4 sets in a Borsuk partition of an extremal set $X$, but no matter how we partition $X$ into subsets of smaller diameter, there would be a point which is blocked from at least three other groups. 

\subsection{Constructing bodies of constant width}
A body of constant width in $\R^d$ is a convex body whose width is 1 in every possible direction. Generally, it is not easy to construct bodies of constant width. 
Meissner bodies are a special family of 3-dimensional bodies of constant width that were introduced in \cite{MON17}. Here we follow the modern treatment by Hynd \cite{hynd2024density}\footnote{That is, we allow any extremal set $X$, not just strongly critical ones.}. Recall that $\Bc(X):=\bigcap_{x\in X} B(x)$ even for an infinite set $X$. Let $X$ be an extremal set on $n\ge 4$ vertices and let $(e_1,e_1^*),(e_2,e_2^*),\dots, (e_{n-1},e_{n-1}^*)$ denote the dual edge pairs of $\skel\Bc(X)$. That is, $e_i$ and $e_i^*$ are circular arcs in $\R^3$. Then a body of the form $\Bc(X\cup e_1\cup\dots \cup e_{n-1})$ is called a \emph{Meissner polyhedron}. In other words, we take the Reuleaux polyhedron $\Bc(X)$ and smooth one edge of each dual edge pair by intersecting the body with unit balls centered on the other edge of the pair. A Meissner polyhedron is always a body of constant width. Furthermore, Meissner polyhedra are dense within the space of constant width bodies in the Hausdorff topology \cite{sallee1970reuleaux,hynd2024density}.

If $X$ forms the corners of a regular tetrahedron, then there are two essentially different ways to smooth the edges. The two resulting shapes are known as Meissner tetrahedra. Our results constructed new extremal sets, hence we obtained new Meissner polyhedra. As the combinatorial structure of the boundary of a Meissner polyhedron is determined by the underlying Reuleaux polyhedron and the choices between dual edges, our method also generates every such cell structure.

 The Blaschke-Lebesgue theorem says that the smallest-area body of constant width in $\R^2$ is the Reuleaux triangle. The  Blaschke–Lebesgue problem asks to solve this question in higher dimensions. Bonnesen and Fenchel conjectured that the two Meissner tetrahedra have the minimal volume among three-dimensional bodies of constant width \cite{bonnesen1948konvexe}.  For some recent results on the 3-dimensional case, see the works of Bogosel \cite{bogosel2024,BOG24,bogosel2026isoperimetric}.

\subsection{Connection to Steinitz's theorem}\label{subs:steinitz}

The classical theorem of Steinitz says that a simple graph is isomorphic to the skeleton of a 3-polytope if and only if it is 3-connected and planar \cite{Steinitz1922}. Steinitz's theorem has a number of different proofs. In \cite{BarnetteGruenbaum1969Steinitz} we find a proof by Barnette and Grünbaum based on the following ideas.
\begin{itemize}
    \item They show that any simple 3-connected graph can be reduced to a wheel graph by a sequence of edge deletion and edge contraction steps while preserving that the graph is simple and 3-connected. Wheel graphs can be realized as pyramids.
    \item They show that if $G'=G-e$ and $G'$ is already realized as a polyhedron, then we can realize $G$ by perturbing the realization of $G'$.  Suppose $e$ has endpoints $u,v$ and incident faces $F_a,F_b$.  In $G'$, the two faces adjacent to $e$ have merged into one face, and $uv$ is only a chord of that face. We want to perturb the realization so that this chord becomes a real edge separating two distinct faces.
    \item They show that there is an ordering of $V\cup F$ such that it starts with $u,v,F_a,F_b$ and every element is incident with at most three earlier ones. The existence of the ordering relies on the 3-connectedness of $G$. Then we can realize $G$ by placing each vertex and face one by one following this order. We can use the first four to realize the edge $e$. Then, each new vertex has to lie on at most three already chosen planes, and a new face has to pass through at most three already chosen points, so the choices are possible in $\R^3$. If the perturbation is sufficiently small, we stay close to the polyhedron of $G'$ and the convexity is preserved, hence the resulting polyhedron realizes $G$. 
    \item The edge contraction step can be performed by taking an edge deletion step on the polar polyhedron. 
\end{itemize}

Our proof of Theorem \ref{thm:main_nbqp_to_diam} is clearly a variant of this proof. The main difference is that we need a reduction step that preserves self-duality. As we discussed in Section \ref{sec:generating}, in a 2-connected strongly involutive self-dual map $G$, other than $K_4$, we can always find a dual edge pair $(e,e^*)$ so that deleting $e$ and contracting $\tau_G(e)$ preserves 2-connectedness and strongly involutive self-duality.  
Then in the perturbation part we must place a vertex and the corresponding face simultaneously. We could hope for an ordering of the vertices such that each vertex forms a diameter with at most three previous vertices and the ordering starts with the four vertices of the split. As we consider 2-connected graphs, this is not always possible. Consider, for example, the split shown in Figure \ref{fig:noordering}. If we start the ordering with $a,b,c,d$ then the last vertex of the ordering is clearly adjacent to four previous vertices. 

\begin{figure}
    \centering
    \includegraphics[width=0.25\linewidth]{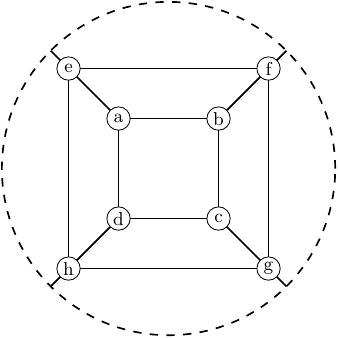}
    \includegraphics[width=0.25\linewidth]{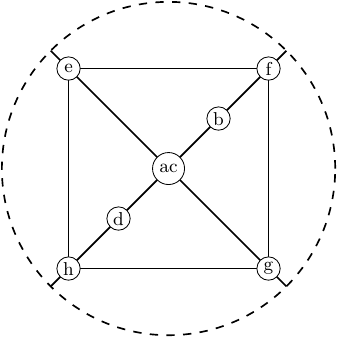}
    \caption{A face contraction without a good ordering of the vertices.}
    \label{fig:noordering}
\end{figure}

This idea could be fixed by showing that given an NBQP map $Q$ there is always a face contraction $abcd$ such that there is a good ordering starting with $a,b,c,d$. We decided to go with a rigidity argument instead, which works for any face contraction.

\section{Open problems}\label{sec:openproblems}

\subsection{Diameter graphs and bodies of constant width with symmetry}
Mani showed that every simple 3-connected planar graph has a representation as the skeleton of a usual convex 3-polytope such that every automorphism of the graph extends to an isometry of the polytope \cite{Mani1971}. Does this extend to diameter graphs and Reuleaux polyhedra? We believe this is the case. We must be careful: in the non-3-connected case there might be graph automorphisms of the graph of $Q$ that are not map automorphisms. Let $Aut(Q)$ denote the set of map automorphisms of $Q$.

\begin{conjecture}
    Suppose $Q$ is a simple non-bipartite quadrangulation map of the projective plane. Then there is an extremal set $X\subset \R^3$ such that the diameter quadrangulation is $Q$, and $Aut(Q)$ acts as a group of isometries on $X$. 
\end{conjecture}

This is equivalent to the following.

\begin{conjecture}
    Suppose $G$ is a 2-connected strongly involutive self-dual map.  Then there is an extremal set $X\subset \R^3$ such that $\skel (\Bc(X))=G$ and $Aut(G)$ acts as a group of isometries on $\Bc(X)$. 
\end{conjecture}

Sallee proved that Meissner polyhedra are dense in the space of bodies of constant width in $\R^3$ with respect to the Hausdorff distance \cite{sallee1970reuleaux, hynd2024density}. Some bodies of constant width have nontrivial symmetry; see, for example, \cite{arman2024small,arman2025small} for bodies of constant width with tetrahedral symmetry. It would be natural to approximate these with Meissner bodies that have the same symmetry group. Unfortunately, this is not always possible, since we modify only one of the edges from each dual-edge pair. For example, it is not hard to see that there is no Meissner body with tetrahedral symmetry. On the other hand Reuleaux polyhedra are also dense in the space of bodies of constant width, and they can exhibit the necessary symmetries. 
\begin{problem}\label{prob:densesym}
    Let $H$ be a symmetry group of a Reuleaux polyhedron. Is the family of Reuleaux polyhedra having symmetry group containing $H$ dense in the family of bodies of constant width having symmetry group $H$?
    
\end{problem}

\subsection{Stress-free diameter frameworks in higher dimensions}
As mentioned in the introduction, diameter graphs are also studied in higher dimensions. Already in dimension 4 we can have quadratically many diameter pairs \cite{erdHos1959sets,swanepoel2009unit}. This also shows that in $\R^4$ not all maximal diameter frameworks are stress-free. Indeed, by \eqref{eq:dimension} we know that a stress-free framework in $\R^4$ with at least 4 vertices has at most $4n-10$ bars. This raises the  natural question of the characterization of stress-free diameter frameworks in $\R^d$. In particular we are interested in the following questions.

\begin{question}
    What is the maximal number of edges in an $n$-vertex stress-free diameter framework in $\R^d$?
\end{question}
\begin{question}
    Suppose $X\subset \R^4$ is a finite set of points in general position, i.e. no five lie on a hyperplane. Is it true that the diameter framework of $X$ is stress-free?
\end{question}

Kalai proposed a number of conjectures that could, in some sense, save the Borsuk conjecture. One of them is the following. 
\begin{conjecture}[Kalai \cite{Kalai_2015}]
    If $G$ is a stress-free diameter framework in $\R^d$  then $G$ is $(d+1)$-colorable.
\end{conjecture}

It would also be interesting to generalize the vertex splitting method to higher dimensions and study how it relates to stress-free diameter frameworks.

\subsection{Tangency graphs of translates}
The diameter graphs can also be understood as tangency graphs of pairwise intersecting unit balls. Indeed, if $\diam (X)=2$, then the diameter graph of $X$ is the same as the tangency graph of the unit balls centered at the points of $X$. Therefore, a graph is the tangency graph of pairwise intersecting unit balls in $\R^3$ if and only if it is a subgraph of an NBQP graph. What can we say about translates of other bodies? 
\begin{question}
Suppose $C$ is a smooth strictly convex body in $\R^3$. Which graphs are tangency graphs of families of pairwise intersecting translates of $C$?
\end{question}

If $C$ is centrally symmetric, we can also understand this question as changing the norm. Given a norm, which graphs are diameter graphs using this norm?

\section*{Acknowledgements}
I thank Andreas Holmsen, Konrad Swanepoel, László Lovász, Tibor Jordán, Márton Naszódi, Máté Jánosik and Dömötör Pálvölgyi for helpful discussions and valuable comments. I am forever grateful to my wife Anna for her continuous support.

\section*{Declaration of AI usage}

\begin{center}
\begin{tabular}{l c}
\hline
\textbf{Activity} & \textbf{AI used} \\
\hline
Asking the research question     & No \\
Coming up with the approach      & No \\
Proof development                & No \\
Writing and exposition           & Yes \\
Checking for mistakes                & Yes \\
Other supporting tasks           & Yes \\
\hline
\end{tabular}
\end{center}

\paragraph{Details.}

The mathematical results are due to the author. That is, every new theorem was originally proved by the author (somewhere around 2024), and a proof for each was originally written by the author. Then, AI was extensively used to clarify the ideas in the paper and to improve the presentation. This includes simplifying some of the original proofs and catching numerous inaccuracies. In some cases, the AI also suggested better definitions and clearer arguments. The AI also helped the author realize that some of the ideas have already appeared in the literature. Most notably, during the conversations with AI the author has learned about the proof of Steinitz's theorem that is now presented in Subsection \ref{subs:steinitz}. Then the author realized that the proof of Theorem \ref{thm:main_skeleton} is a direct analog of that proof. AI help was also used for creating some of the figures, especially the ones in the Appendix.

\bibliographystyle{plain} 
\bibliography{sample}

\appendix
\section{Examples}\label{app:examples}

In this appendix, we collect several examples. Each shows a non-bipartite quadrangulation of the projective plane $Q$ and the corresponding strongly involutive self-dual map $G$.  We also include figures for a diameter realization of $Q$ and a Reuleaux polyhedron with skeleton $G$. See the webpage \cite{damasdi_selfdual_visualizations} for more examples.

\begin{figure}[htbp]
    \centering

    \begin{minipage}[t]{0.48\textwidth}
        \centering
        \includegraphics[width=0.7\linewidth]{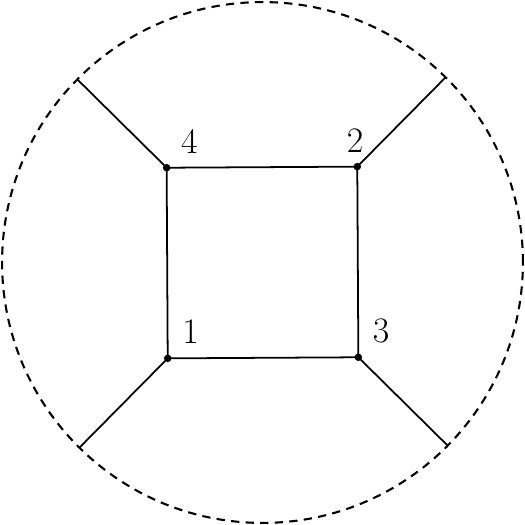}
        \par\smallskip
        (a) NBQP map
    \end{minipage}
    \hfill
    \begin{minipage}[t]{0.48\textwidth}
        \centering
        \includegraphics[width=0.7\linewidth]{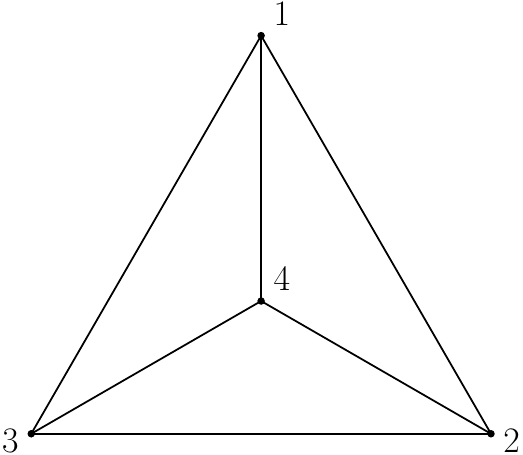}
        \par\smallskip
        (b) SISD map
    \end{minipage}
    
    \medskip
    
    \begin{minipage}[t]{0.48\textwidth}
        \centering
        \includegraphics[width=0.6\linewidth]{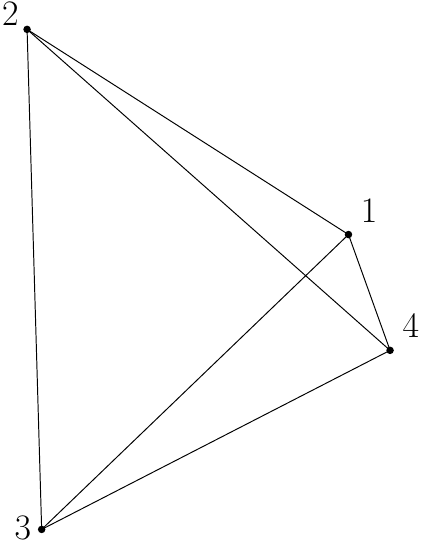}
        \par\smallskip
        (c) Diameter framework
    \end{minipage}
    \hfill
    \begin{minipage}[t]{0.48\textwidth}
        \centering
        \includegraphics[width=0.7\linewidth]{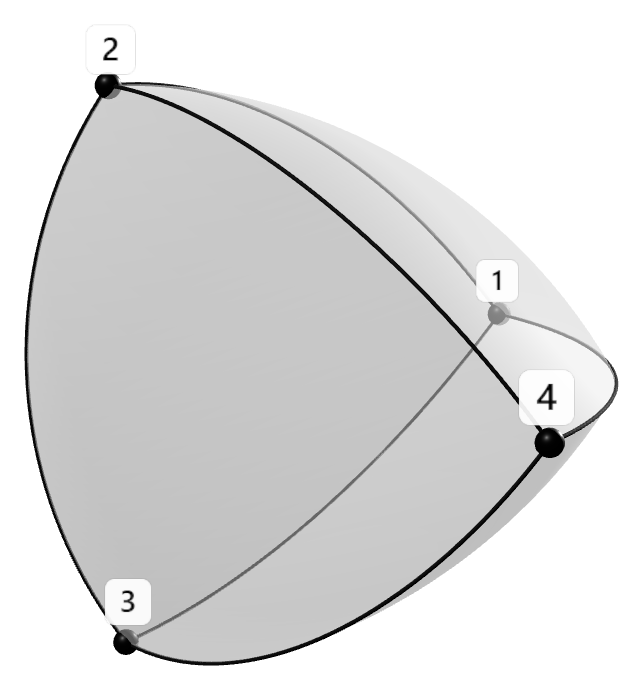}
        \par\smallskip
        (d) Reuleaux polyhedron
    \end{minipage}
    \refstepcounter{example}
    \label{fig:example-A}
    \caption*{Example \theexample. The unique NBQP and SISD map on 4 vertices. The two graphs are isomorphic as abstract graphs, but this is just happenstance. }
\end{figure}

\begin{figure}[htbp]
    \centering

    \begin{minipage}[t]{0.48\textwidth}
        \centering
        \includegraphics[width=0.7\linewidth]{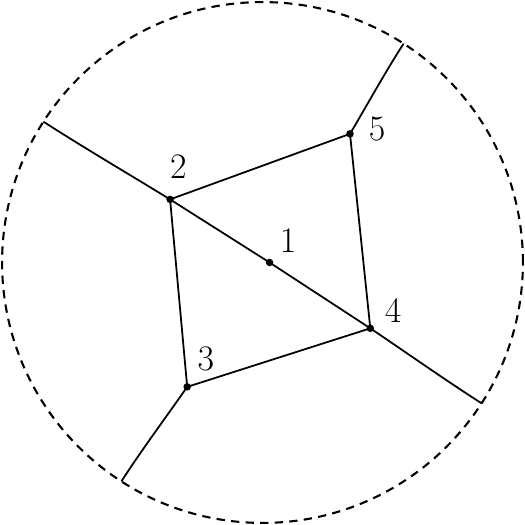}
        \par\smallskip
        (a) NBQP map
    \end{minipage}
    \hfill
    \begin{minipage}[t]{0.48\textwidth}
        \centering
        \includegraphics[width=0.7\linewidth]{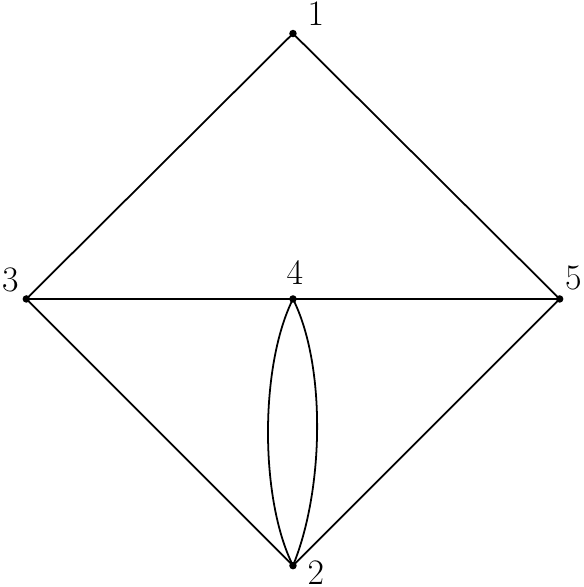}
        \par\smallskip
        (b) SISD map
    \end{minipage}

    \medskip

    \begin{minipage}[t]{0.48\textwidth}
        \centering
        \includegraphics[width=0.7\linewidth]{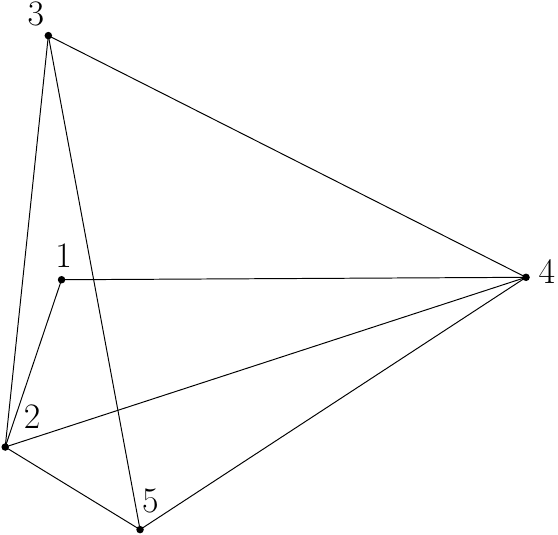}
        \par\smallskip
        (c) Diameter framework
    \end{minipage}
    \hfill
    \begin{minipage}[t]{0.48\textwidth}
        \centering
        \includegraphics[width=0.7\linewidth]{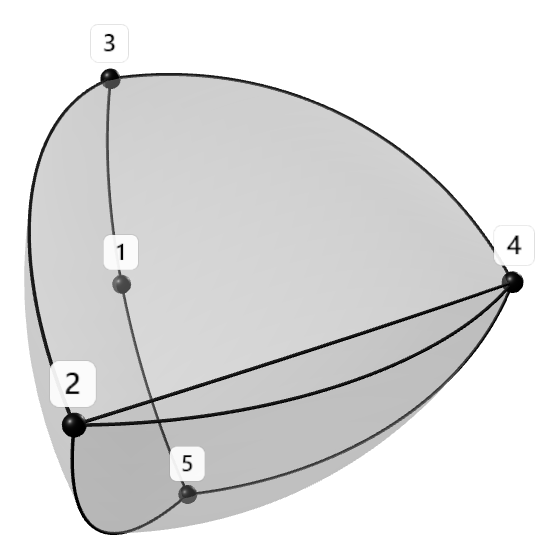}
        \par\smallskip
        (d) Reuleaux polyhedron
    \end{minipage}

    \refstepcounter{example}
    \label{fig:example-B}
    \caption*{Example \theexample. The unique NBQP and SISD map on 5 vertices. The SISD map and Reuleaux body have parallel edges between 2 and 4.}
\end{figure}

\begin{figure}[htbp]
    \centering

    \begin{minipage}[t]{0.48\textwidth}
        \centering
        \includegraphics[width=0.7\linewidth]{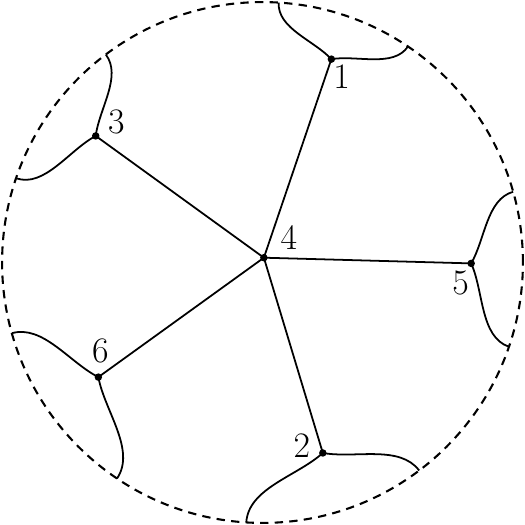}
        \par\smallskip
        (a) NBQP map
    \end{minipage}
    \hfill
    \begin{minipage}[t]{0.48\textwidth}
        \centering
        \includegraphics[width=0.7\linewidth]{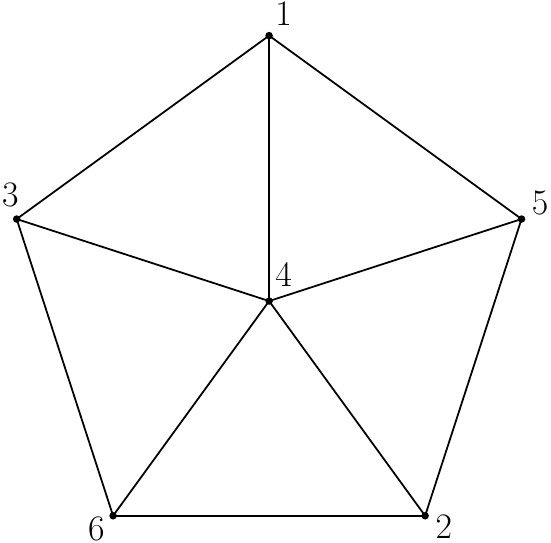}
        \par\smallskip
        (b) SISD map
    \end{minipage}

    \medskip

    \begin{minipage}[t]{0.48\textwidth}
        \centering
        \includegraphics[width=0.7\linewidth]{figcollection/cikk-n6-3-diam.pdf}
        \par\smallskip
        (c) Diameter framework
    \end{minipage}
    \hfill
    \begin{minipage}[t]{0.48\textwidth}
        \centering
        \includegraphics[width=0.7\linewidth]{figcollection/cikk-n6-3-reu.png}
        \par\smallskip
        (d) Reuleaux polyhedron
    \end{minipage}

    \refstepcounter{example}
    \label{fig:example-C}
    \caption*{Example \theexample. The unique separating 4-cycle free NBQP on 6 vertices and the corresponding 3-connected SISD map. The two graphs are isomorphic as abstract graphs, but this is just happenstance.}
\end{figure}

\begin{figure}[htbp]
    \centering

    \begin{minipage}[t]{0.48\textwidth}
        \centering
        \includegraphics[width=0.7\linewidth]{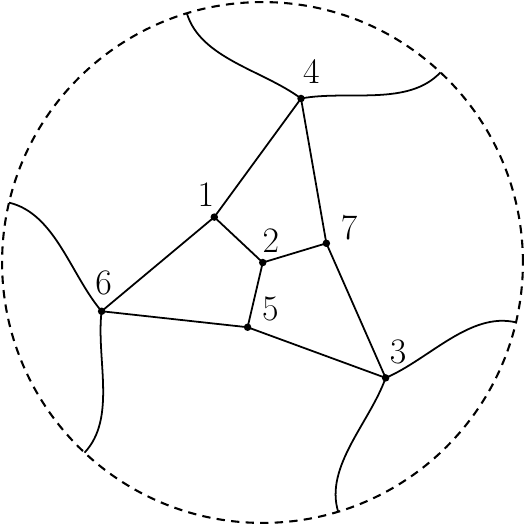}
        \par\smallskip
        (a) NBQP map
    \end{minipage}
    \hfill
    \begin{minipage}[t]{0.48\textwidth}
        \centering
        \includegraphics[width=0.7\linewidth]{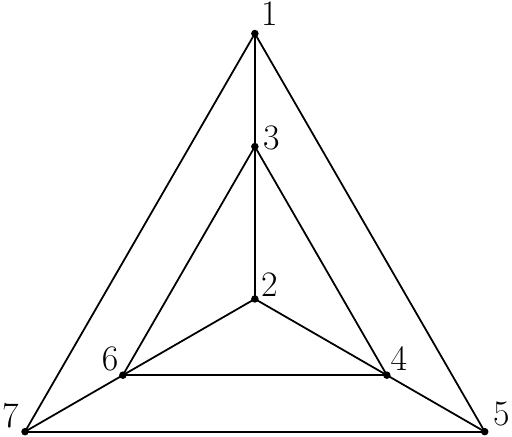}
        \par\smallskip
        (b) SISD map
    \end{minipage}

    \medskip

    \begin{minipage}[t]{0.48\textwidth}
        \centering
        \includegraphics[width=0.7\linewidth]{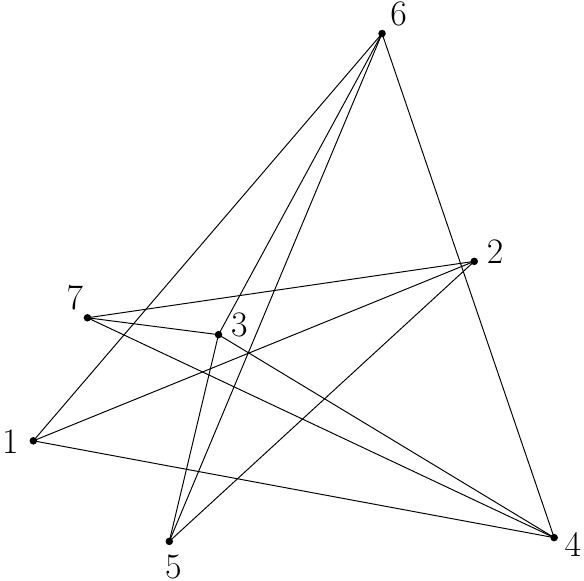}
        \par\smallskip
        (c) Diameter framework
    \end{minipage}
    \hfill
    \begin{minipage}[t]{0.48\textwidth}
        \centering
        \includegraphics[width=0.7\linewidth]{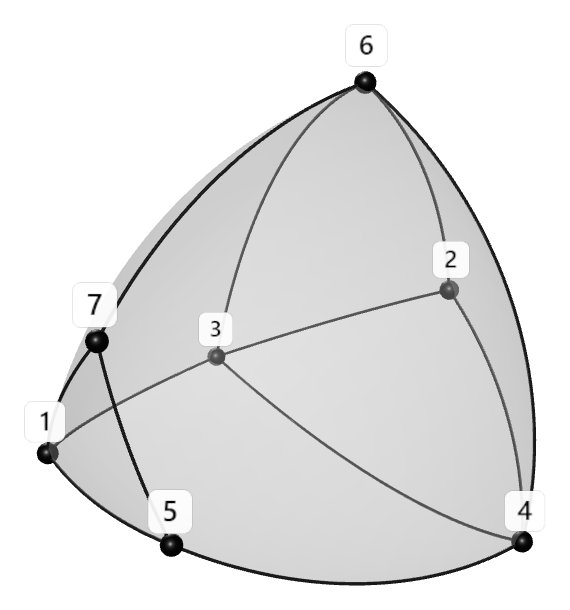}
        \par\smallskip
        (d) Reuleaux polyhedron
    \end{minipage}

    \refstepcounter{example}
    \label{fig:example-D}
    \caption*{Example \theexample. The unique separating 4-cycle free NBQP on 7 vertices and the corresponding 3-connected SISD map.}
\end{figure}

\begin{figure}[htbp]
    \centering

    \begin{minipage}[t]{0.48\textwidth}
        \centering
        \includegraphics[width=0.7\linewidth]{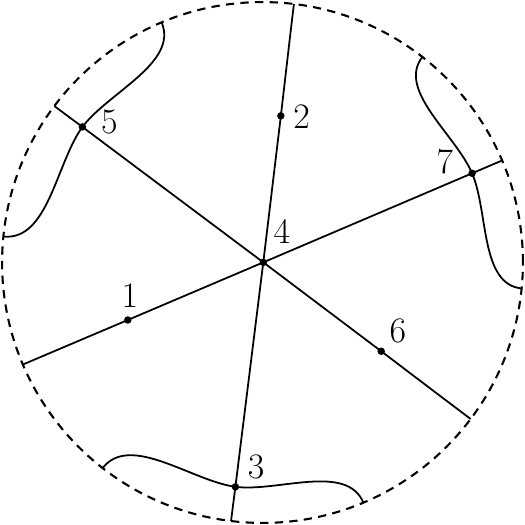}
        \par\smallskip
        (a) NBQP map
    \end{minipage}
    \hfill
    \begin{minipage}[t]{0.48\textwidth}
        \centering
        \includegraphics[width=0.7\linewidth]{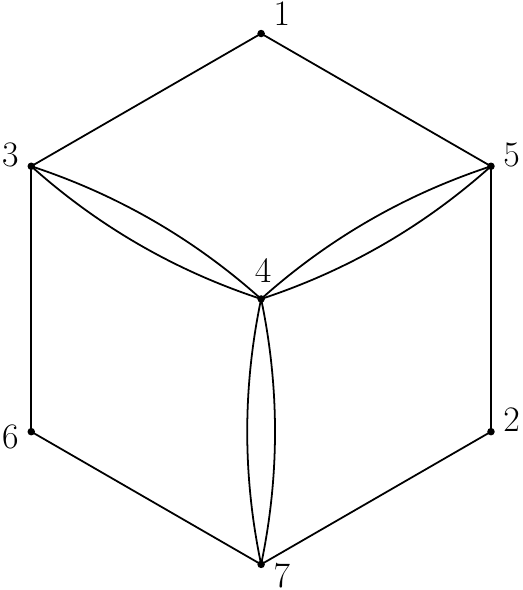}
        \par\smallskip
        (b) SISD map
    \end{minipage}

    \medskip

    \begin{minipage}[t]{0.48\textwidth}
        \centering
        \includegraphics[width=0.7\linewidth]{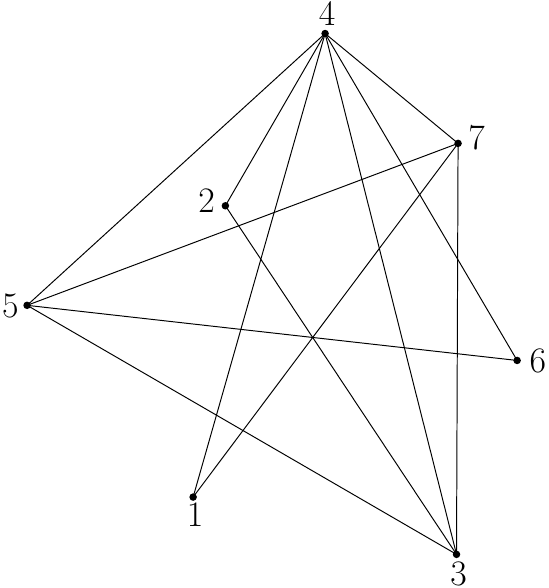}
        \par\smallskip
        (c) Diameter framework
    \end{minipage}
    \hfill
    \begin{minipage}[t]{0.48\textwidth}
        \centering
        \includegraphics[width=0.7\linewidth]{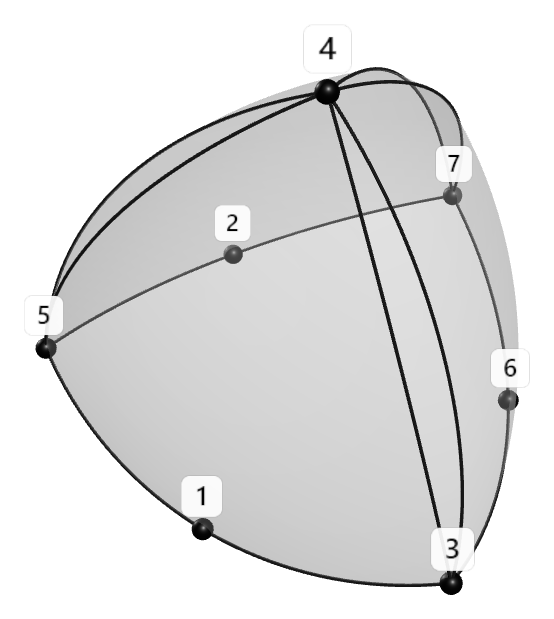}
        \par\smallskip
        (d) Reuleaux polyhedron
    \end{minipage}

    \refstepcounter{example}
    \label{fig:example-E}
    \caption*{Example \theexample. Another example on 7 vertices.}
\end{figure}

\begin{figure}[htbp]
    \centering

    \begin{minipage}[t]{0.48\textwidth}
        \centering
        \includegraphics[width=0.7\linewidth]{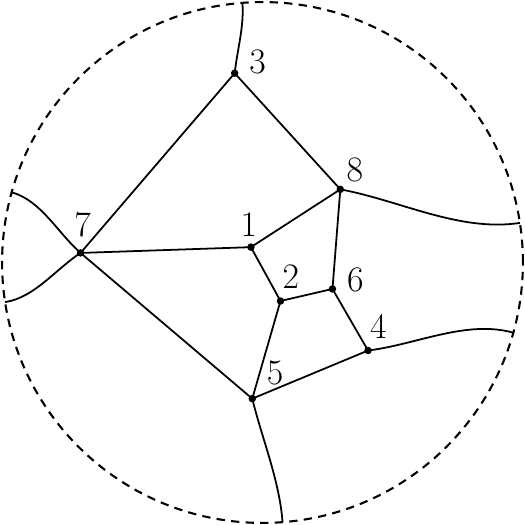}
        \par\smallskip
        (a) NBQP map
    \end{minipage}
    \hfill
    \begin{minipage}[t]{0.48\textwidth}
        \centering
        \includegraphics[width=0.7\linewidth]{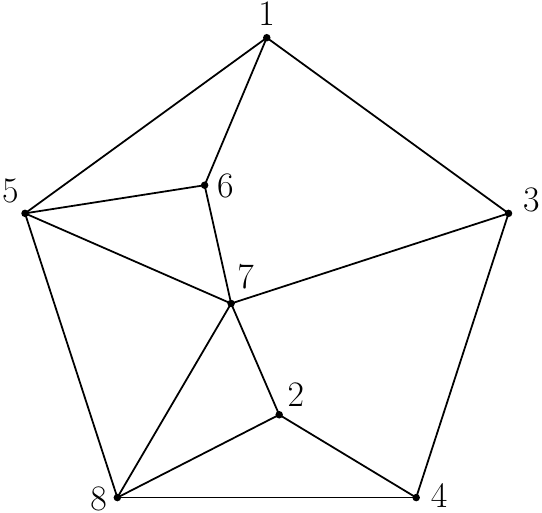}
        \par\smallskip
        (b) SISD map
    \end{minipage}

    \medskip

    \begin{minipage}[t]{0.48\textwidth}
        \centering
        \includegraphics[width=0.6\linewidth]{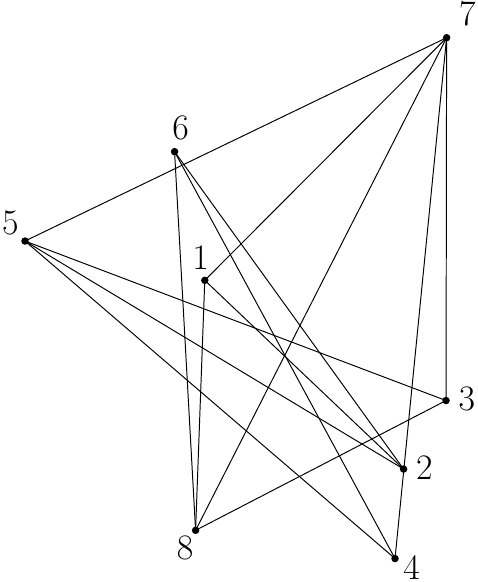}
        \par\smallskip
        (c) Diameter framework
    \end{minipage}
    \hfill
    \begin{minipage}[t]{0.48\textwidth}
        \centering
        \includegraphics[width=0.7\linewidth]{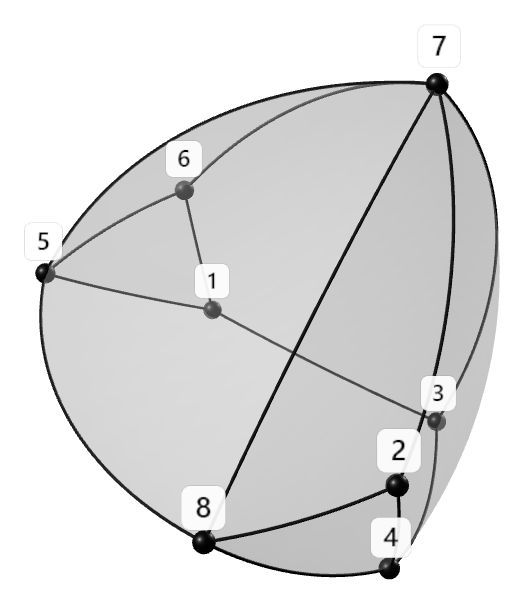}
        \par\smallskip
        (d) Reuleaux polyhedron
    \end{minipage}

    \refstepcounter{example}
    \label{fig:example-F}
    \caption*{Example \theexample. One of the two separating 4-cycle free NBQPs on 8 vertices. The other one is a wheel graph.}
\end{figure}

\begin{figure}[htbp]
    \centering

    \begin{minipage}[t]{0.48\textwidth}
        \centering
        \includegraphics[width=0.7\linewidth]{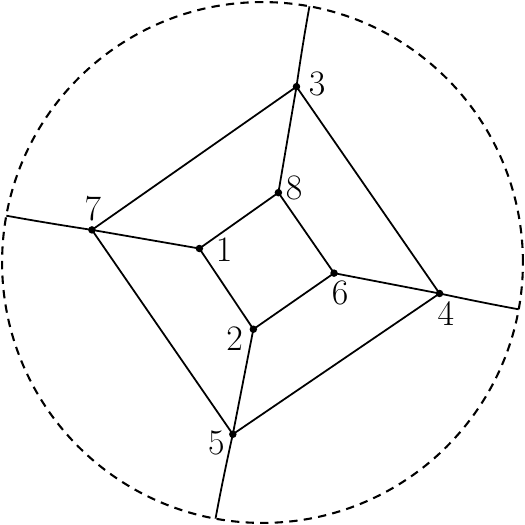}
        \par\smallskip
        (a) NBQP map
    \end{minipage}
    \hfill
    \begin{minipage}[t]{0.48\textwidth}
        \centering
        \includegraphics[width=0.7\linewidth]{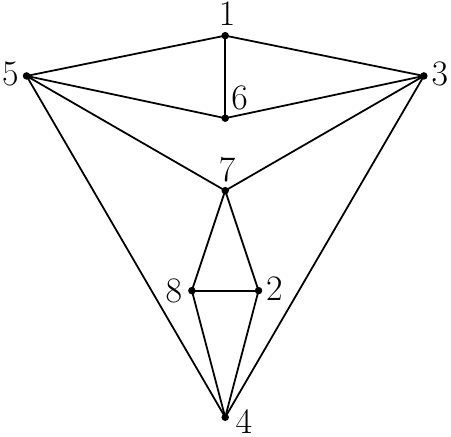}
        \par\smallskip
        (b) SISD map
    \end{minipage}

    \medskip

    \begin{minipage}[t]{0.48\textwidth}
        \centering
        \includegraphics[width=0.6\linewidth]{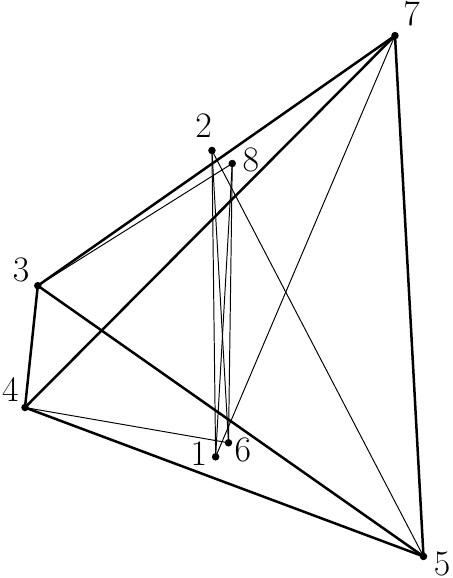}
        \par\smallskip
        (c) Diameter framework
    \end{minipage}
    \hfill
    \begin{minipage}[t]{0.48\textwidth}
        \centering
        \includegraphics[width=0.7\linewidth]{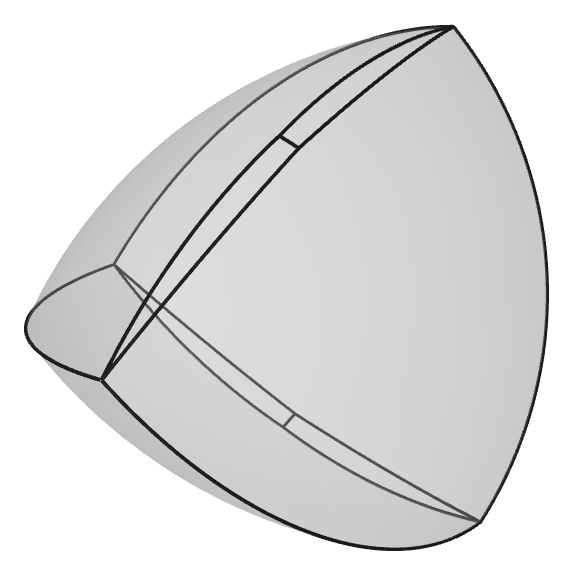}
        \par\smallskip
        (d) Reuleaux polyhedron
    \end{minipage}

    \refstepcounter{example}
    \label{fig:example-G}
    \caption*{Example \theexample. The smallest NBQP where the minimum degree is at least 3, but we do have a separating 4-cycle. The corresponding SISD map does not have parallel edges, but it is not 3-connected. The diameter framework is critical but not strongly critical.}
\end{figure}

\begin{figure}[htbp]
    \centering

    \begin{minipage}[t]{0.48\textwidth}
        \centering
        \includegraphics[width=0.7\linewidth]{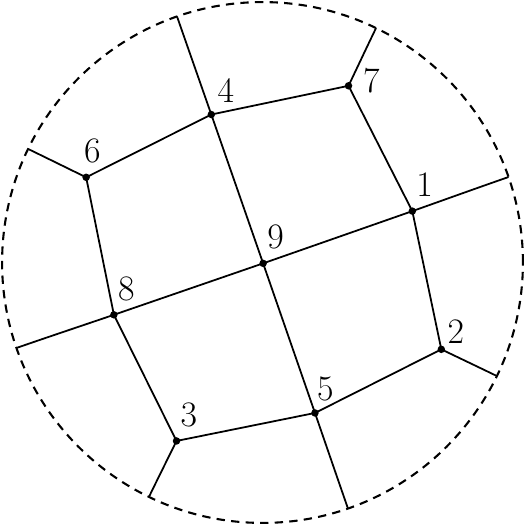}
        \par\smallskip
        (a) NBQP map
    \end{minipage}
    \hfill
    \begin{minipage}[t]{0.48\textwidth}
        \centering
        \includegraphics[width=0.7\linewidth]{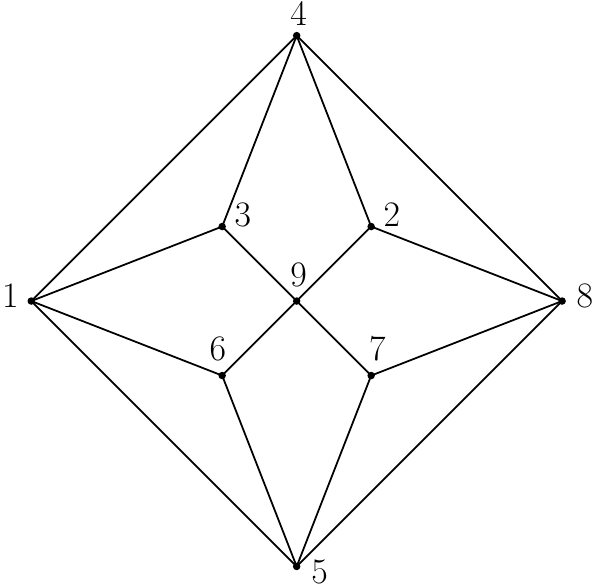}
        \par\smallskip
        (b) SISD map
    \end{minipage}

    \medskip

    \begin{minipage}[t]{0.48\textwidth}
        \centering
        \includegraphics[width=0.6\linewidth]{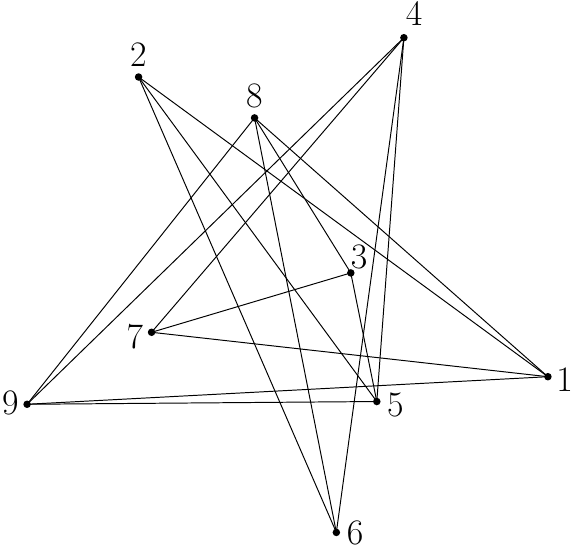}
        \par\smallskip
        (c) Diameter framework
    \end{minipage}
    \hfill
    \begin{minipage}[t]{0.48\textwidth}
        \centering
        \includegraphics[width=0.7\linewidth]{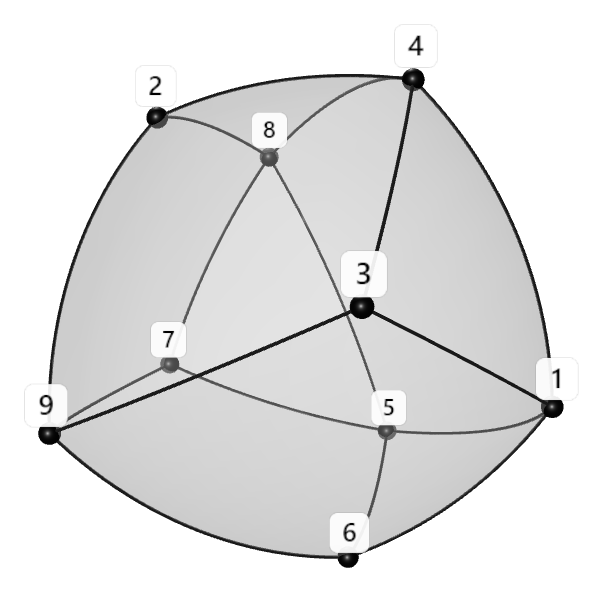}
        \par\smallskip
        (d) Reuleaux polyhedron
    \end{minipage}

    \refstepcounter{example}
    \label{fig:example-H}
    \caption*{Example \theexample. A highly symmetric example on 9 vertices.}
\end{figure}

\end{document}